\documentclass[final,onefignum,onetabnum]{siamart171218}

\usepackage{lipsum}
\usepackage{amsfonts}
\usepackage{graphicx}
\usepackage{epstopdf}
\usepackage{subfig}
\usepackage{enumerate}
\usepackage{algorithmic}
\ifpdf
  \DeclareGraphicsExtensions{.eps,.pdf,.png,.jpg}
\else
  \DeclareGraphicsExtensions{.eps}
\fi

\newsiamremark{remark}{Remark}
\newsiamremark{hypothesis}{Hypothesis}
\crefname{hypothesis}{Hypothesis}{Hypotheses}
\newsiamthm{claim}{Claim}

\headers{Totally Positive Hessenberg Toeplitz matrices}{N. Ercolani, J. Peca-Medlin, J. Ramalheira-Tsu}

\title{Total positivity and spectral theory for Toeplitz Hessenberg matrix ensembles}

\author{Nicholas Ercolani\thanks{Department of Mathematics, University of Arizona, Tucson, AZ 85721 USA 
  (\email{ercolani@math.arizona.edu}).}
\and John Peca-Medlin\thanks{Department of Mathematics, University of Arizona, Tucson, AZ 85721 USA 
  (\email{johnpeca@math.arizona.edu}.}
\and Jonathan Ramalheira-Tsu\thanks{Department of Mathematics, University of Arizona, Tucson, AZ 85721 USA 
  (\email{jramalheiratsu@arizona.edu}.}
  }

\usepackage{amsopn}
\DeclareMathOperator{\diag}{diag}

\makeatletter
\newcommand*{\addFileDependency}[1]{
  \typeout{(#1)}
  \@addtofilelist{#1}
  \IfFileExists{#1}{}{\typeout{No file #1.}}
}
\makeatother

\ifpdf
\hypersetup{
  pdftitle={Totally Positive Hessenberg Matrices},
  pdfauthor={N. Ercolani, J. Peca-Medlin, and J. Ramalheira-Tsu}
}
\fi

\begin{document}

\maketitle

\begin{abstract}
  In this paper we define and lay the groundwork for studying a novel matrix ensemble: totally positive Hessenberg Toeplitz operators, denoted TPHT. This is the intersection of two ensembles that have been significantly explored: totally positive Hessenberg matrices (TPH) and Hessenberg Toeplitz matrices (HT). TPHT has a rich linear algebraic and spectral structure that we describe. Along the way we find some previously unnoticed connections between certain Toeplitz normal forms for matrices and Lie theoretic interpretations. We also numerically study  the spectral asymptotics of TPH matrices via the TPHT  ensemble and use this to open a study of TPHT with random symbols.
\end{abstract}



\section{Introduction}

 This paper concerns a particular subclass of Hessenberg matrices  and their spectral asymptotics. The $n \times n$ (lower) Hessenberg matrices, $\mathcal{H}$, in general take the form  
\begin{eqnarray*}
\left(\begin{array}{ccccc}
* & 1 & & &\\
*& * & 1 & &\\
\vdots & \ddots & \ddots & \ddots &\\
\vdots &  & \ddots & \ddots & 1\\
* & \dots & \dots & * & *
\end{array} \right).
\end{eqnarray*}

The first subclass we consider is the totally positive Hessenberg matrices (TPH). We follow a standard convention here by taking totally positive  (TP) to mean that all minors are non-negative. (If we mean to say that all minors are positive we will refer to this as being {\it strictly} TP.) This enables us to characterize other sparsity patterns as being TP, such as lower triangular TP matrices. 
 
 Hessenberg matrices themselves have played an elemental role in many areas of linear and numerical linear algebra. For instance, every matrix is conjugate via a Householder reflector to a Hessenberg matrix, called its Hessenberg form. The Hessenberg form is preserved by the QR decomposition, and so is an essential component of many tools used to compute eigenvalues and eigenvectors, including the QR algorithm, Lanczos Iterations, and generalized minimal residual method (GMRES) \cite{Fr61,Fr62,Ku61,La50,SaSc86}. Additionally, Hessenberg matrices are essential  in Neville elimination, an alternative to Gaussian elimination to find an LU decomposition that  iteratively zeros out each subdiagonal and so maintains an upper Hessenberg form that moves toward the final upper triangular factor \cite{Gasca1995}.
 
 More recently there has been a particular focus on the TPH ensembles we consider here. These bring to bear tools from other areas of mathematics such as network theory and dynamical systems theory. TPH ensembles have played a fundamental role in studying the analytical combinatorics of networks because they are generalizations of path-counting matrices \cite{FZ00}.  Consequently, this class of matrices has natural coordinatizations stemming from Whitney-Loewner factorization (see Theorem \ref{thm:wp}). That in turn has applications to the dynamics of 
LU factorizations as well as to integrable Toda lattices and their generalizations. In particular this makes a connection to recent work on integrable systems theory and analytical combinatorics. In \cite{ER23}, we find simultaneous TPH realizations of the integrable Full Toda lattice \cite{EFS93} 
on different space-time scales. For discrete space-time we realize a novel combinatorial interpretation of the LU algorithm in terms of the dynamics of extended box-ball  systems.  In related work \cite{Fukuda13} 
Fukuda and others have made use of Full Toda lattices and formal connections to orthogonal polynomials to try to develop improved eigenvalue algorithms for totally positive Hessenberg (TPH) matrices. This appears to have potential connections to the rigorous analysis of bi-orthogonal polynomials developed in \cite{EM01}. 
Other applications, by Demmel and Koev, concern high relative accuracy for eigenvalue calculations (see \cite{DeKo05,Ko07}).

In another direction there has been a recent focus on a different class of matrices that is both Hessenberg and Toeplitz (HT) with many applications in linear algebra related for example to orthogonal polynomials, stochastic filters, time series analysis and difference approximations to initial-BVP problems for PDE \cite{GKS72}. 

In this paper we begin to analyze questions that lie naturally at the intersection of these two classes: totally positive Hessenberg Toeplitz matrices (TPHT). The key point for our work here is that each isospectral class of Hessenberg matrices contains a
unique Toeplitz matrix. This allows us to bring forward and apply powerful tools from Toeplitz theory to investigate spectral questions for general TPH matrices. This ties into another more recent and principal motivation for this work which concerns the study of integrable systems evolving on spaces of random or rough data. That work seeks to gain insights into dynamics of more general conservative evolution equations in random environments (see \cite{Spohn21} 
for a general survey). Past models have focused on the classical Toda whose phase space is tridiagonal Hessenberg matrices with independent random entries. The recent work in \cite{ER23} 
suggests how these studies may be extended to general TPH ensembles with appropriate random entries. These motivations will be further described in Section \ref{numspec}, but our overall goal here is to take a step toward showing that such studies may be reduced to considering spectra of random TPHT class. Along the way we uncover some novel aspects of the linear algebra underlying this class.
\bigskip

The outline of this paper is as follows. In Section \ref{sec:background}  we present the essential background for the two fundamental classes, HT and TPH, on which this work is based. In particular  we review some of the relevant remarkable properties of TP matrices and illustrate their realization within the TPH ensemble. We then introduce the novel aspects of the intersection ensemble, TPHT, of TPH and HT. Along the way we describe results on HT normal forms for general Hessenberg matrices. We show that these normal forms are very naturally related to more general normal forms in Lie theory  originally due to Kostant. We believe this is the first time this connection has been noticed and we make use of it in later sections as well as relating it to other natural normal forms (detailed in the Appendix). Finally we  discuss 
connections of the HT normal form to LU factorization, providing also the explicit LU form for TPHT matrices in \Cref{thm:LU-T}.

In Section \ref{thm:GS} we review the Grenander-Szeg\H{o} theory that provides the principal tool for understanding spectral asymptotics within TPH.

Section \ref{numspec} motivates our numerical study of the spectral theory of TPH and how to access this through TPHT. The bulk of the section presents random realizations of the spectral asymptotics.

In Section \ref{conclusions}, we describe a number of potential applications for our work. Finally, in \Cref{App} we detail the aforementioned interplay between  various normal forms for the TPHT ensemble while \Cref{sec:proofthm2} contains the detailed proof for \Cref{thm:LU-T}.
\bigskip

\section{Background and Motivation}
\label{sec:background}
\subsection{Hessenberg-Toeplitz Normal Form}
\label{sec:HT}

Toeplitz matrices are distinguished by having constant values along diagonals.
More precisely, an $n \times n$ matrix $X$ is Toeplitz if there are $2n-1$ numbers, 
$x_{-n + 1}, \dots, x_0, \dots, x_{n-1}$
such that the $(i,j)$ coefficient of $X$ is $x_{j-i}$ for $1 \leq i,j \leq n$. 

Mackey, Mackey and Petrovic derived the following elegant, constructive result in \cite{MaMaPe99}.  First recall that a {\it nonderogatory} matrix is defined as one all of whose eigenspaces are one-dimensional, meaning that each eigenvalue corresponds to one and only one Jordan block. Also let $M_n(\mathbb{R})$ denote the space of
of $n \times n$ matrices over $\mathbb{R}$. 
\begin{theorem}[\cite{MaMaPe99}] \label{thm:mackey}
Every nonderogatory element of $M_n(\mathbb{R})$ is similar to a {\it unique} Hessenberg-Toeplitz (HT)
matrix. Alternatively, every nonderogatory isospectral class contains a unique HT matrix. 
\end{theorem}
Since elements of the space of $n \times n$ Hessenberg matrices, denoted $\mathcal{H}_n$,  are all nonderogatory (see Proposition 1f of 
\cite{MaMaPe99}), one has the following: 
\begin{corollary} \label{cor:iso}
   Every isospectral class in $\mathcal{H}_n$ with respective (possibly repeated) eigenvalues $\Lambda = \{ \lambda_1,\dots \lambda_n\}$, and denoted $\mathcal{O}_\Lambda$, contains a unique Toeplitz  matrix. 
\end{corollary}

\begin{remark}
   This explicit result is an instance of a more general, but not constructive, Lie theoretic result due to Kostant \cite{K78}. 
In essence, this states that for the analogue of Hessenberg matrices in a semi-simple Lie algebra, there is a cross-section of the isospectral classes such that elements of a given isospectral class are conjugate to a unique element of the cross-section with the conjugation given by a unique
   lower unipotent matrix. In our case that cross-section is given by Toeplitz matrices.
   This is more precisely stated in Appendices \ref{App} and \ref{sec:proofthm2} where other, constructive, normal forms 
   of potential interest to us 
   are also presented. 
\end{remark}

These results and their applications provide a strong motivation for our study of the HT class. The focus of this paper is to study aspects of a class with the further restriction of being totally positive, the TPHT class.
\bigskip

For later use we introduce here the notion of the {\it symbol} of a Toeplitz operator, $T$, which is the
  Taylor-Laurent series (or Fourier series) whose $k^{th}$ coefficient is taken to be the constant value along the $k^{th}$ diagonal of a bi-infinite Toeplitz matrix. In the application for this paper we will be concerned with symbols that correspond to polynomials of Hessenberg type, meaning of the form $\varphi(T) = t^{-1} (1 + x_1 t + \dots +x_n t^n)$ (or trigonometric polynomials in the Fourier presentation).  Toeplitz matrices are then formed by taking finite size truncations of the associated   bi-infinite matrix. For more  on the characterization of Toeplitz operators in the bi-infinite  setting we refer the reader to the seminal paper of Aissen, Edrei, Schoenberg and Whitney  \cite{AESW51}. See also remarks in the Conclusions.

\subsection{Totally Positive Hessenberg matrices} \label{TPH}
TP matrices themselves have a rich structure, which is nicely described in Ando's survey  \cite{Ando}. One of the most salient of these are spectral oscillation properties that generalize classical Perron-Frobenius results for positive matrices. More precisely one has 

\begin{theorem}[\cite{Ando}] \label{thm:osc}
    If $A$ is an $n \times n$ strictly TP matrix then all its eigenvalues are real, distinct and positive.
    Let $\mathbf{u}_k$ denote the (real) eigenvector corresponding to the $k^{th}$ eigenvalue (in descending ordered), then $\mathbf{u}_k$ has exactly $k-1$ variations of sign. Moreover, the nodes of ${u}_k(t)$  and those of ${u}_{k+1}(t)$ are interlacing. 
\end{theorem}
By a {\it variation of sign} here we mean consecutive entries in the eigenvector where the sign changes. For 
${u}_k(t)$ we have the following general definition. For a vector $\mathbf{x} \in \mathbb{R}^n$ define the piece-wise linear function $x(t)$ for $1\leq t \leq n$ by
\[
x(t) = (k+1-t) x_k + (t-k) x_{k+1} \quad \mbox{for} \,\,\,
k \leq t \leq k+1.
\]
The nodes of $x(t)$ are the roots of $x(t) = 0$.

If $A$ is TP but not strictly so, then it is in the closure of strictly TP matrices (see Theorem 2.7 of \cite{Ando}). In particular, if $A$ has full rank with distinct eigenvalues (which will be the case for the Hessenberg matrices we will consider) then the stated results of the theorem continue to hold but with the possibility that some zero crossings might coalesce at an intermediate node.

\begin{definition}
    We will denote the cases of $n \times n$ Toeplitz matrices with (factored) symbol $\varphi(T) = t^{-1} \prod_{j=1}^m (1+a_jt)^m$ by  $({\textbf a_m},n)$. Let $\textbf 1_m$ denote the vector in $\mathbb R^m$ whose components are all $1$. When $m=n$ we just write ${\textbf a_n}$ or $\textbf 1_n$ for the associated Toeplitz matrix. 
\end{definition}

We illustrate all this in the following examples along with \Cref{fig:1n} for the case of   Toeplitz matrices with symbols $t^{-1}(1+t)^5$ and $t^{-1}(1+t)^{10}$, where we display the relevant information in two panels for $\textbf 1_5$ 
and $\textbf 1_{10}$, respectively. In the first case \eqref{T5 evals} shows the  eigenvalues ordered by decreasing size. Then
\eqref{T5 evects} displays the matrix of the associated eigenvectors in corresponding order. Finally Figure \ref{fig:osc_1n_5} shows the piece-wise linear interpolations of the eigenvectors, oriented vertically to the correspondence with the eigenvectors in \eqref{T5 evects}. Bars are included in this figure to mark where the zero-crossings occur. 
A similar set of panels is shown for $\textbf 1_{10}$ culminating in Figure \ref{fig:osc_1n_10}.

    For the case of symbol $t^{-1}(1+t)^5$, whose coefficients would then align with the standard binomial coefficients,  the matrix truncated to size $5$ is as follows. 
    \begin{align} \label{T5}
        \textbf 1_5 &= \begin{bmatrix}
            5   &  1  &   0  &   0 &    0\\
        10 &    5&     1  &   0  &   0\\
        10  &  10 &    5   &  1   &  0\\
         5   & 10  &  10    & 5    & 1\\
         1    & 5   & 10    &10     &5
        \end{bmatrix} = V\Lambda V^{-1},
    \end{align}
    where $\Lambda$ and $V$ denote, respectively, the eigenvalue and eigenvector matrices, with approximate computed forms of
    \begin{align}
        \Lambda &\approx \operatorname{diag}(11.0024, 
        7.9317, 
        4.3187, 
        1.5285, 
        0.2187), \label{T5 evals}\\
        V &\approx \begin{bmatrix}
            0.0047  & -0.0065&    0.0105&   -0.0182&   -0.0258\\
        0.0283   &-0.0190   &-0.0071 &   0.0632 &   0.1234\\
        0.1226   & 0.0091 &  -0.0999  & -0.0373  & -0.3319\\
        0.4061   & 0.2822  &  0.0347   &-0.3203   & 0.6111\\
        0.9051   & 0.9591   & 0.9943   & 0.9443   &-0.7075
        \end{bmatrix} \label{T5 evects}
    \end{align}

    \begin{figure}[t] \label{Fig1}
\centering
\subfloat[$\textbf 1_{5}$]{%
    \includegraphics[width=.6\textwidth, angle = 90]{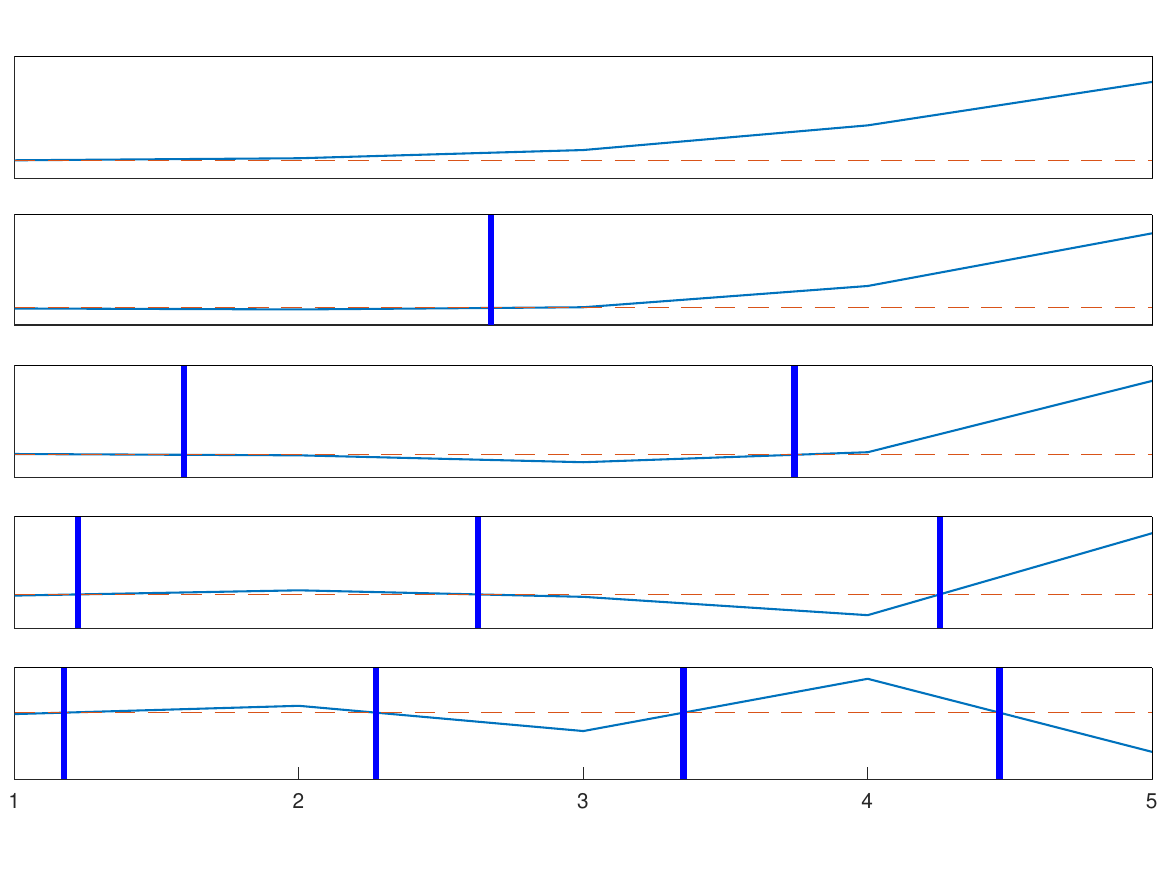}%
    \label{fig:osc_1n_5}%
    }%
\subfloat[$\textbf 1_{10}$]{%
    \includegraphics[width=.6\textwidth, angle = 90]{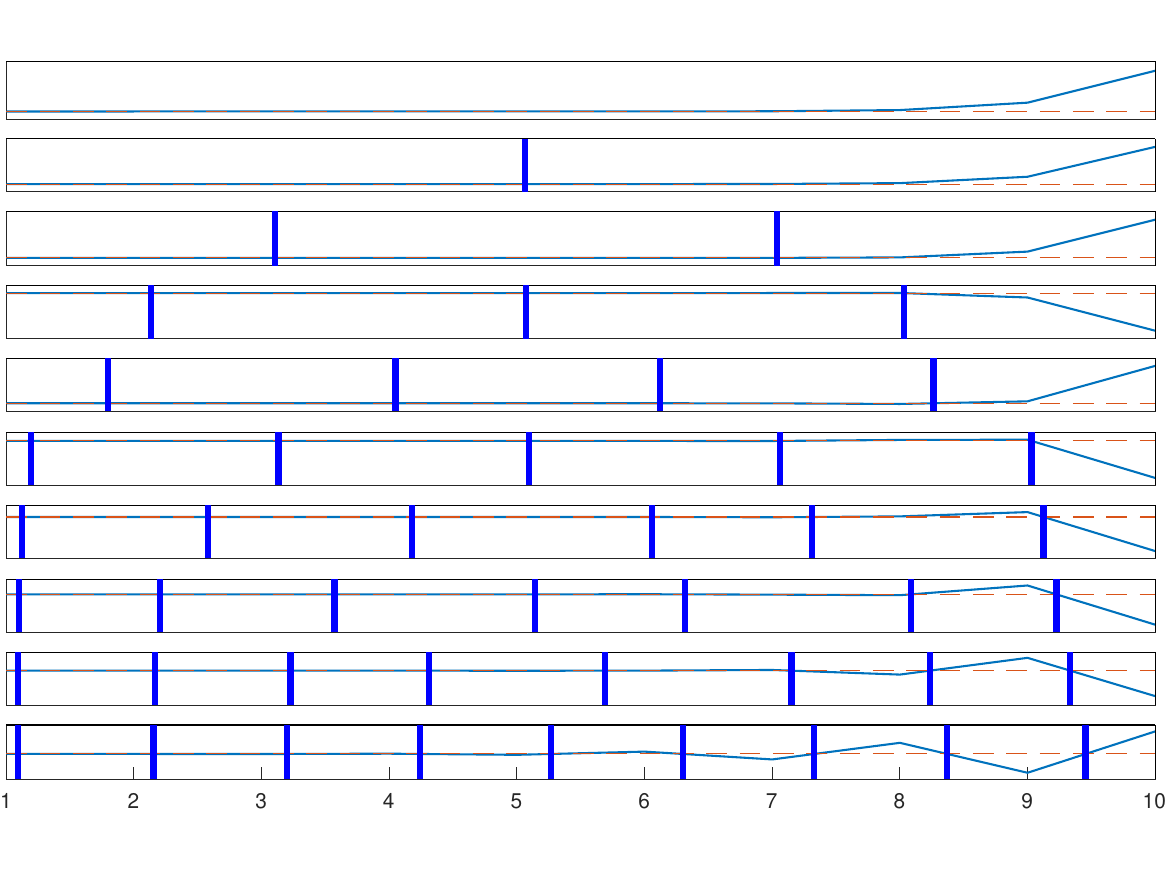}%
    \label{fig:osc_1n_10}%
    }%
\caption{Oscillating eigenvector interlacing zero maps for $n\times n$ matrix $\textbf 1_n$, formed using $n = 5,10$, ordered in decreasing associated eigenvalue size, with dotted red line indicating the standard $x$-axis and blue bars indicating associated zero crossings.}
\label{fig:1n}
\end{figure}

    
    One may directly check that \eqref{T5} is TP. To further illustrate the TP properties of the eigenvalues and eigenvectors of \eqref{T5}, we have that the (computed) eigenvalues are strictly positive while the associated (computed) eigenvector matrix  illustrates the sign variation property of the theorem. In summary, the oscillatory properties we exemplify in this example include:
    \begin{enumerate}[i]
        \item {Perron-Frobenius theorem}: There is an eigenvector for the largest eigenvalue with positive components. This is clearly seen in the leftmost displays of and Figures \ref{fig:osc_1n_5} and \ref{fig:osc_1n_10} as well as by the first column of \eqref{T5 evects}. 
    
\item Each successive eigenvector (with respect to the ordering of the eigenvalues from largest to smallest) introduces an additional sign change from the prior step. For instance in  \eqref{T5 evects}  one sees that the entries in the leftmost eigenvector are all positive; in the next eigenvector there is one sign change  between the second and third entry; in the third eigenvector there are two sign changes, one between the first and second entries and another between the third and fourth entries; and so on.    This is  the variation of signs stated in the theorem.

        \item The piecewise linear maps that connect the points $(j, v_{k,j})$  for successive eigenvectors $\mathbf{v}_k$ have interlacing zeros (nodes of the
        $v_k(t)$ described in the theorem.)  This interlacing is clear from the interlacing of the consecutive bar configurations in Figures 
        \ref{fig:osc_1n_5} and \ref{fig:osc_1n_10}. 
    \end{enumerate}


It is natural to seek a characterization of TPH matrices
along the lines of what was described for Toeplitz matrices. A step in that direction was carried out in \cite{ER23}, 
which approaches this in terms of LU factorization. This essentially follows from the Whitney-Loewner theorem \cite{FZ00}
. The result is 

\begin{theorem}[\cite{ER23}] \label{thm:wp}
Let $\mathcal{B}$ denote the subvariety of $\mathcal{H}_n$ comprised  of upper bidiagonal matrices of the form 
\begin{eqnarray*}
\left[\begin{array}{cccc}
* & 1\\
& *  & \ddots\\
&&\ddots & 1\\
&&& * 
\end{array}\right]
\end{eqnarray*}
and let $\mathcal{B}^{\geq 0}$ denote the submanifold in which all diagonal entries are positive. Then the subvariety, TPH, of totally positive Hessenberg matrices has the decomposition     
\begin{eqnarray} \label{resphase}
\mathcal{H}^{\geq 0} = (N_-^{\geq 0} \times \mathcal{B}^{\geq 0}~) 
\end{eqnarray}
where the superscript, $\{\geq 0\}$, denotes total positivity.  ($N_-$ here denotes the 
space of $n \times n$ lower unipotent
matrices.) 
\end{theorem}

As is further discussed in \cite{ER23}
, the LU decomposition described in the theorem can be iterated to define a dynamic process on $\mathcal{H}_n$. This is an isospectral process that can be used to approximate eigenvalues by iteratively computing the $A = A^{(0)} = LU$ factorization of an input matrix and then inverting the order of the factors for $A^{(1)} = UL = UA^{(0)}U^{-1}$; this is followed by iteratively computing the LU factorization of $A^{(i+1)}$ from input $A^{(i)}$. TP matrices have the additional property that each LU factor is itself TP while also the product of TP matrices is TP \cite{Ando,Cryer76}. So each intermediate matrix in the iterated LU algorithm is TP. Since the process is isospectral, the eigenvalues never change. The eigenvectors do change; however, they maintain the oscillatory properties, stated in Theorem \ref{thm:osc}, throughout. For example,  \Cref{fig:osc_1n_5_step10,fig:osc_1n_10_step10} show the oscillating eigenvector interlacing zero maps for the 10$^{th}$ iterate  of the LU map (denoted $\textbf 1_n^{(10)}$) using $\textbf 1_n$ for $n = 5,10$ (cf. Figures \ref{fig:osc_1n_5} and \ref{fig:osc_1n_10}). In Appendix \ref{App} an explicit expression for the eigenvectors is given in terms of normal forms 
(see Corollary \ref{cor:efcn}).

\begin{figure}[t] \label{Fig2}
\centering
\subfloat[$\textbf 1_{5}^{(10)}$]{%
    \includegraphics[width=.6\textwidth, angle = 90]{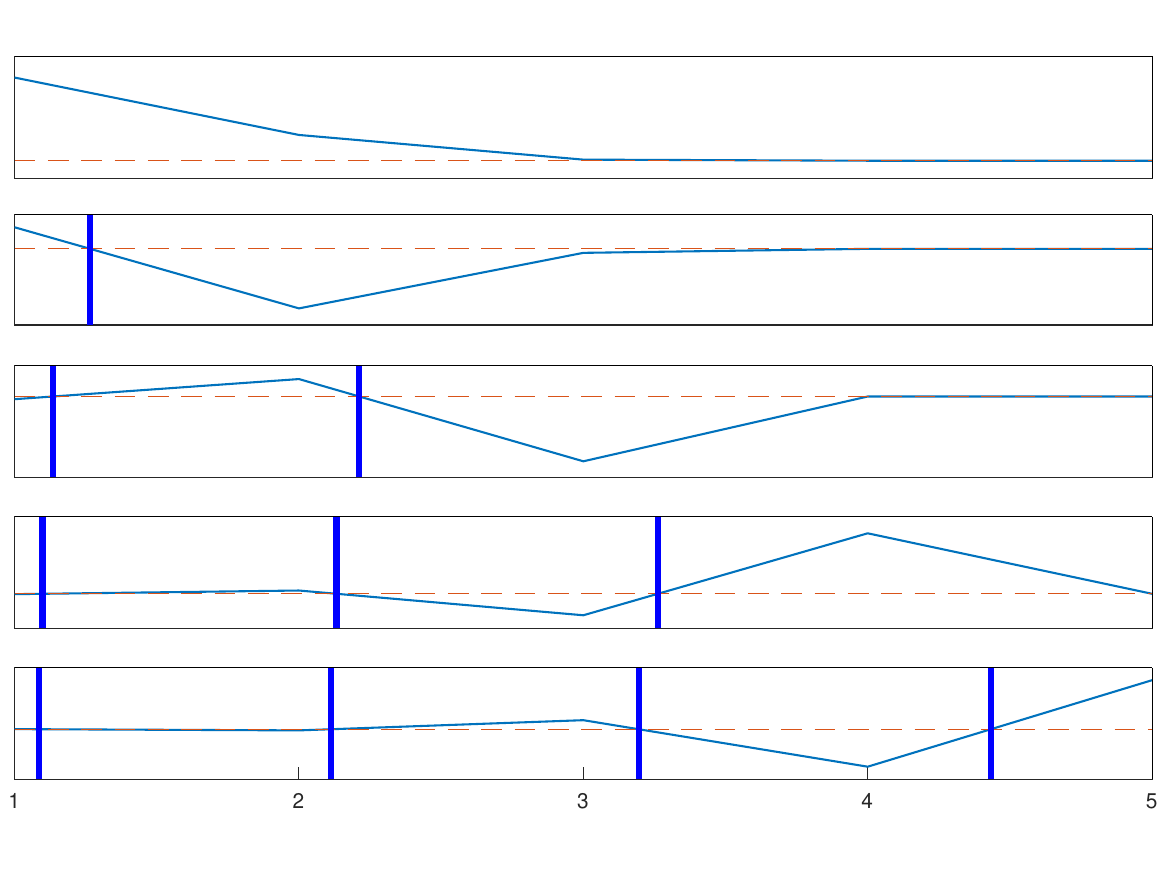}%
    \label{fig:osc_1n_5_step10}%
    }%
\subfloat[$\textbf 1_{10}^{(10)}$]{%
    \includegraphics[width=.6\textwidth, angle = 90]{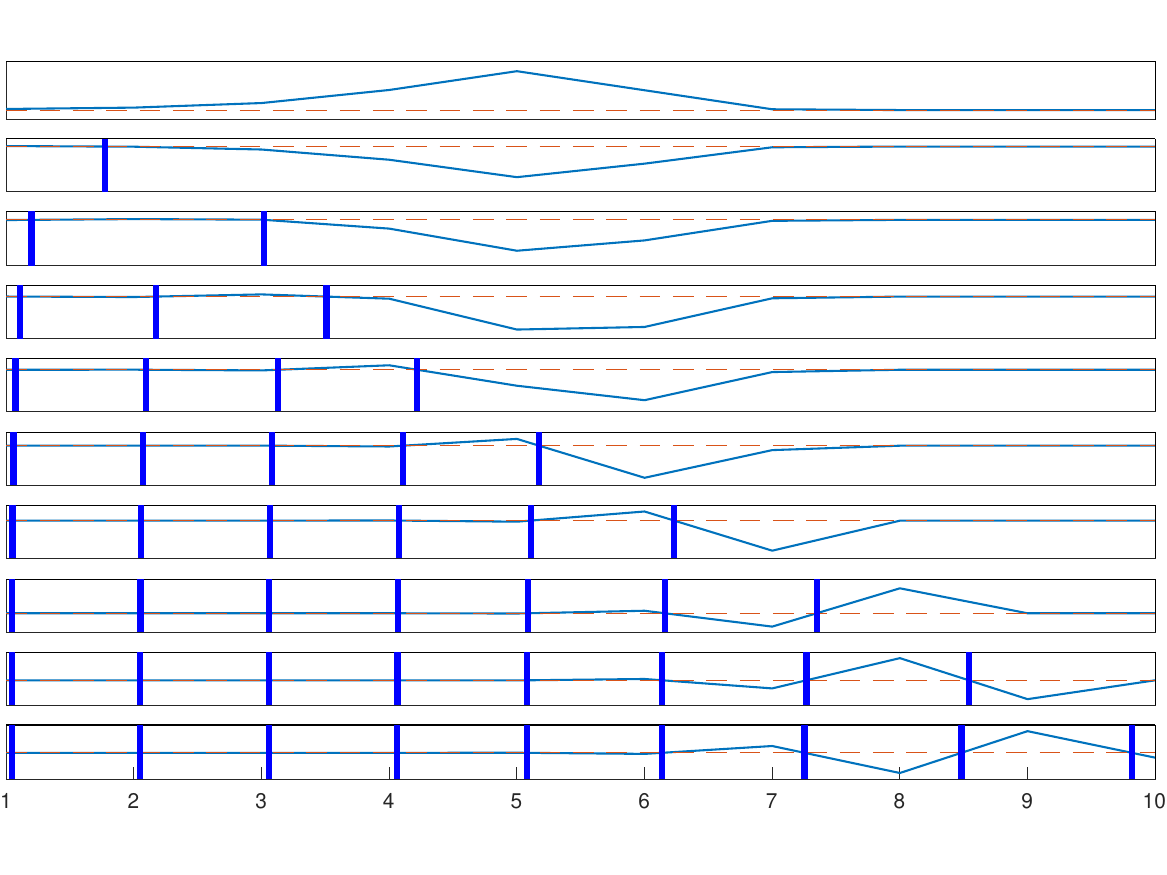}%
    \label{fig:osc_1n_10_step10}%
    }%
\caption{Oscillating eigenvector interlacing zero maps for $n\times n$ matrix $\textbf 1_n^{(10)}$, formed using $n = 5,10$.}
\end{figure}

\subsection{The Hessenberg-Toeplitz Normal form and Total Positivity }

TPH matrices are, of course, a subclass of the general class of Hessenberg matrices and in particular they form a subclass within each isospectral class, $\mathcal{O}_\Lambda$,  of Hessenberg matrices {with fixed spectrum $\Lambda$}. This raises the natural question of whether the unique Toeplitz matrix of Corollary \ref{cor:iso} is TP, i.e. an element of TPHT. (Of course in these TP cases one should restrict attention to $\Lambda$ with all eigenvalues non-negative.)

The Hessenberg Toeplitz (HT) operators we consider are finite banded with symbols, as defined in Section \ref{sec:HT}, of the (factored) form

\begin{equation} \label{rootedsymb}
\varphi(T)(z) = z^{-1}  \prod_{\ell=1}^n
(1 + a_\ell z).
\end{equation}
We note that  
this HT symbol amounts to an upward shift of the diagonals for operators corresponding to symbols that are polynomial and, therefore, whose associated bi-infinite matrix operator  is lower triangular. Since the $a_\ell$ are non-negative {in the TPHT case},  it then follows (see Theorem \ref{thm:ET}) that the corresponding $n\times n$ truncations, $T_n$, of these {TP}HT operators have non-negative minors and so are TP. Hence, the $n \times n$ TPHT matrices depend on $n$ real, non-negative parameters, the $a_\ell$.

We then have the following key result.

\begin{corollary} \label{cor:open}
TPHT is the closure of an open set within the class of HT matrices and therefore represents the closure of
an open set of isospectral equivalence classes in $\mathcal{H}_n$. 
\end{corollary}

The first statement follows because the class of HT $n \times n$ matrices is $n$-dimensional, as is TPHT. The rest of the statement follows from 
Corollary \ref{cor:iso}. By Kostant's theorem (cf. Appendix \ref{App}) the isospectral classes in $\mathcal{H}_n$ are conjugacy classes under the adjoint action of $N_-$ on $\mathcal{H}_n$. Hence, by Corollary \ref{cor:iso}, we have a 1:1 correspondence
\begin{eqnarray*}
   \mathcal{H}_n/N_- &\to &  HT \\
   \mathcal{O}_\Lambda &\mapsto & T_\Lambda
\end{eqnarray*}
 defined by mapping the isospectral class to the unique HT matrix it contains, denoted by $T_\Lambda$. Since, by Corollary \ref{cor:iso}  
 this  is 1:1 and TPHT is the closure of an  open subset  of HT, the latter statement of
 Corollary \ref{cor:open} follows.

So TPHT is a robust and natural class to study.

\subsection{LU Factorization in the HT Ensemble} \label{LUfact}
We pause here to discuss how one may identify the unique HT operator within a given isospectral class of $\mathcal{H}_n$.
For this we can make use of the LU decomposition described in Theorem \ref{thm:wp}.

For convenience of notation, in the following definition and theorem, if $n\in\mathbb{N}$, denote by $[n]$ the set $\{1,2,\ldots,n\}$. If $n<1$, we shall take $[n]$ to be the empty set.

\begin{definition}
Let $A$ be an $n\times n$ matrix and $S\subseteq \{1,2,\ldots,n\}$. If $S$ is non-empty, denote by $\tau_S^\text{init}(A)$ the minor for the associated sub-matrix of $A$ with columns given by the initial $|S|$ columns of $A$ and rows indexed by $S$:
$$\tau_S^\text{init}(A) = \det\left(A_{S,[|S|]}\right).$$
If $S$ is the empty set, we take this to be $1$. 
\end{definition}

\begin{theorem} \label{thm:LU-T}
Let $T$ be an $n\times n$ TPHT matrix {in $\mathcal{H}^{\geq 0}$, defined in Theorem \ref{thm:wp}.}  Then $T$ has the LU decomposition $T=LU$ where
\begin{align*}
    (L)_{ij} &= 
    \left\{
    \begin{array}{cc}
0 & i<j\\
\frac{\tau_{\{i\}\cup[j-1]}^\text{init}(T)}{\tau_{[j]}^\text{init}(T) } & i\geq j
    \end{array}
    \right.\\
    (U)_{ij} &= 
    \left\{
    \begin{array}{cc}
\frac{\tau_{[i]}^\text{init}(T)}{\tau_{[i-1]}^\text{init}(T)} & i=j\\
1 & j=i+1\\
0 & \text{otherwise.}
    \end{array}
    \right.
\end{align*}
\end{theorem}
\begin{proof}[Sketch of proof:]
We provide here just a sketch of the full proof which can be found in Appendix \ref{sec:proofthm2}. The key property used is the explicit form of LU decompositions in the class of Hessenberg matrices. Take the LU decomposition of an $(n+1)\times (n+1)$ Hessenberg matrix:
$$
\left[\begin{array}{c|c}
L_{n+1}^{(n)} & \mathbf{0}_n\\
\hline
\mathbf{p}^T & 1
\end{array}\right]
\left[\begin{array}{c|c}
U_{n+1}^{(n)} & \mathbf{q}\\
\hline
\mathbf{0}_n^T & r
\end{array}\right]
=
\left[\begin{array}{c|c}
L_{n+1}^{(n)}U_{n+1}^{(n)} & L_{n+1}^{(n)}\mathbf{q}\\
\hline
\mathbf{p}^T U_{n+1}^{(n)} & \mathbf{p}^T\mathbf{q}+ r
\end{array}\right],$$
where $L_{n+1}^{(n)}$ and $U_{n+1}^{(n)}$ are both $n\times n$. This then says that respective principal submatrices of the lower and upper matrices of a Hessenberg matrix then themselves constitute an LU decomposition of the corresponding principal submatrix of the Hessenberg matrix.\\[5pt]
The proof we provide leverages this fact to prove Theorem \ref{thm:LU-T} inductively, with the induction step amounting to solving for the unknowns $\mathbf{p}$, $\mathbf{q}$ and $r$. Solving for these unknowns and recognizing the resulting conditions for Theorem \ref{thm:LU-T} to be true as cofactor expansions allows the proof to be completed.
\end{proof}

In \cite{ER23} an alternative parameterization of TPH matrices, due to Lusztig, is employed. 
This is given in terms of a further factorization of $L$ of the form 

\begin{eqnarray} \label{LusFac}
L = (1 + \alpha_1 f_{h_1}) \cdots (1 + \alpha_M f_{h_M})
\end{eqnarray}
where $L \in N^{\ge0}_-$, $M = \binom{n}2$,  $h_j \in \{1, \dots, n \}$, $\alpha_j \in \mathbb{R}_{>0}$, 1 denotes the identity matrix and $f_i$ is the elementary lower matrix with 1 in the 
$(i+1, i)$ entry and zero elsewhere. The choice and ordering of the $h_i$ is determined by a rule described in \cite{ER23}. In this way an element of TPH, with given eigenvalues, may be uniquely decomposed into a product of bidiagonal matrices. Thus the $\alpha_j$ provide an alternative parameterization of TPH. We will not make much mention of this parameterization in the present paper but illustrate here what this decomposition looks like in the case of a $3\times 3$ TPHT matrix, in terms of the coefficient parameters of the Toeplitz symbol.

Let $T$ be the following TPHT matrix
\begin{align}
    T = \left[\begin{array}{ccc}
a_1 & 1   & 0\\
a_2 & a_1 & 1 \\
a_3 & a_2 & a_1
\end{array}\right] .
\end{align}

\noindent First, decompose $T$ using Theorem \ref{thm:LU-T}:

\begin{align}
    T = \left[\begin{array}{ccc}
a_1 & 1   & 0\\
a_2 & a_1 & 1 \\
a_3 & a_2 & a_1
\end{array}\right] = \left[\begin{array}{ccc}
1 & 0 & 0\\
\frac{a_2}{a_1} & 1 & 0\\
\frac{a_3}{a_1} & \frac{a_1a_2-a_3}{a_1^2 - a_2} & 1
\end{array}\right]
\left[\begin{array}{ccc}
a_1 & 1 & 0\\
0 &  \frac{a_1^2-a_2}{a_1} & 1\\
0 & 0 & \frac{a_1^3-2a_1a_2-a_3}{a_1^2-a_2}
\end{array}\right] = LU.
\end{align}
Now, we decompose the lower piece $L$ further:
\begin{equation}
    L = \left[\begin{array}{ccc}1 & 0 & 0\\\frac{a_2}{a_1}-\frac{a_1a_2-a_3}{a_1^2-a_2}&1&0\\0&0&1\end{array}\right]
\left[\begin{array}{ccc}1 & 0 & 0\\0&1&0\\0&\frac{a_1a_2-a_3}{a_1^2-a_2}&1\end{array}\right]
\left[\begin{array}{ccc}1 & 0 & 0\\\frac{a_3(a_1^2-a_2)}{a_1(a_1a_2-a_3)}&1&0\\0&0&1\end{array}\right]
.
\end{equation}

In terms of determinants, and writing $\tau_S$ for $\tau_S^\text{init}(T)$, this decomposition takes the following form:
\begin{align}
    T&=\left[\begin{array}{ccc}
1 & 0 & 0\\
\frac{\tau_{\{2\}}\tau_{\{1,2\}}-\tau_{\{1\}}\tau_{\{1,3\}}}{\tau_{\{1\}}\tau_{\{1,2\}}} & 1 & 0\\
0 & 0 & 1
\end{array}\right]
\left[\begin{array}{ccc}
1 & 0 & 0\\
0 & 1 & 0\\
0 & \frac{\tau_{\{1,3\}}}{\tau_{\{1,2\}}} & 1
\end{array}\right]
\left[\begin{array}{ccc}
1 & 0 & 0\\
\frac{\tau_{\{3\}}\tau_{\{1,2\}}}{\tau_{\{1\}}\tau_{\{1,3\}}} & 1 & 0\\
0 & 0 & 1
\end{array}\right]\\
&\hspace{5pc} \times 
\left[\begin{array}{ccc}
\tau_{\{1\}} & 1 & 0\\
0 & \frac{\tau_{\{1,2\}}}{\tau_{\{1\}}} & 1\\
0 & 0 & \frac{\tau_{\{1,2,3\}}}{\tau_{\{1,2\}}}
\end{array}\right].
\end{align}
For further information on this, we refer the reader to 
\cite{ER23}.

\section{Grenander-Szeg\H{o} Theorem}
\label{thm:GS}
For the asymptotic spectral analysis of elements in TPHT we will use an application of the classical Grenander-Szeg\H{o} theorem. First, recall the empirical spectral distribution (ESD) of a square matrix $A \in M_n(\mathbb C)$ is given by
\begin{equation} \label{esd}
    \mu_A = \frac1n \sum_{k=1}^n \delta_{\lambda_k(A)}.
\end{equation}
This denotes a probability measure that gives equal weight (with multiplicity) to all eigenvalues of $A$. If $A$ is a random matrix, then $\mu_A$ is a random probability measure. As established in \cite{GS58}:

\begin{theorem}[\cite{GS58}]
\label{thm:GSmom}

Let $\varphi(t)$ be the symbol of a Toeplitz operator $T$  with $T_n = T_n(\varphi)$ being the 
$n\times n$ truncation of $T$ with eigenvalues $\lambda_k$. Then 
\begin{equation}
\lim_{n\to\infty}\frac1n \sum_{k=1}^n \left(\lambda_k(T_n) \right)^p = \frac1{2\pi} \int_0 ^{2\pi} (\varphi(e^{i\theta}))^p d\theta.
\end{equation}
\end{theorem}
For our purposes we will take $T$ to be an $(m+1)$-banded HT operator for which the symbol will be
\[
\varphi(t) = t^{-1} (1 + a_1 t) \cdots (1 + a_m t).
\]
Then from Theorem \ref{thm:GSmom}, recast in terms of Cauchy's integral formula, one has
\begin{eqnarray} \label{pmom}
 \lim_{n \to \infty}\frac1n \sum_{k=1}^n \left(\lambda_k(T_n) \right)^p 
 &=& \frac{1}{2 \pi} \int_0^{2 \pi} (\varphi(e^{i\theta}))^p d\theta\\
 &=& \oint_{\mathbb S^1} \varphi(z)^p \frac{dz}{z} \\
 &=& \oint_{\mathbb S^1} z^{-p} 
 [(1 + a_1 z) \cdots (1 + a_m z)]^p \frac{dz}{z}\\
 &=& [z^p]((1 + a_1 z) \cdots (1 + a_m z))^p\\
 &=& [z^{p}] \left\{ \left(\sum_{j=1} ^p \binom{p}j z^j a_1^j \cdots  \sum_{j=1} ^p\binom{p}j z^j a_m^j \right) \right\}\\ \label{coeffs}
 &=&\sum_{\sum i_j = p} \binom{p}{i_1} \cdots \binom{p}{i_m} a_1^{i_1} \cdots
 a_m^{i_m} 
\end{eqnarray}
where $[z^p]$ denotes the coefficient of $z^p$ in the expression that follows.

Theorem \ref{thm:GSmom}  has an extension from matrix moments to any function $f(z)$ that is analytic on a neighborhood of the the convex hull of the essential spectrum of the associated bi-infinite Toeplitz matrix \cite{BS98}. 
\begin{theorem}[\cite{BS98}] \label{thm:BS} For $f$ an entire function,
\begin{equation}
\lim_{n\to\infty}\frac1n \sum_{k=1}^n f\left(\lambda_k(T_n(\varphi)) \right) = \frac1{2\pi} \int_0 ^{2\pi} f(\varphi(e^{i\theta})) d\theta.
\end{equation}
\end{theorem}

Using \Cref{thm:BS} one may study asymptotic limits of moment generating functions or characteristic functions. For instance, consider the discrete measure, from \eqref{esd}, given by 
\[
d\mu_n = \frac1n \sum_{j=1}^n \delta_{\lambda_j^{(n)}}(s). 
\]
Then the moment generating function of this measure is given by
\[
\int e^{ts} d \mu_n = \frac1n \sum_{j=1}^n \exp \left[ t  \left(\lambda_j(T_n) \right) \right],
\]
which in the large $n$ limit approaches
\[
F(t) = \frac1{2\pi i } \oint e^{t \varphi(z)}\frac{dz}{z}.
\]
If $F(t)$ is in the domain of the inverse Laplace transform then one may use this to try to recover an asymptotic density {for the sequence $d\mu_n$}. We show how this goes for a special case in the next sub-section. But in general establishing that this transform exists may prove challenging.  Nevertheless, in principal one may use the moment calculations of Theorem \ref{thm:GSmom} to explicitly calculate Taylor series approximations of $F(t)$. In Section \ref{numspec} we will study the form of these moments in a random setting. 
We now consider some explicit examples.

\subsection{TPHT using $(\textbf 1_m,n)$}
\label{subsec:1mn}
Recall the coefficients in the symbol are formed using the elementary symmetric polynomials for the input parameters. Using the vector of all 1s, $\textbf 1_m$, it follows then the coefficients take the  form $x_k(\textbf 1_m) = \binom{m}k$. Hence $T_n = (\textbf 1_m,n)$ has diagonals consisting of the binomial coefficients. \vspace{1pc}

For $\widehat \varphi(k-1) = \binom{m}{k}$ for $k=0,\ldots,m$, then $T_n$ has {associated} symbol $$
\varphi(z) = z^{-1} \sum_{k=0}^m \binom{m}k z^k = z^{-1}(z+1)^m.
$$

\subsubsection{TPHT using $(\textbf 1_2,n)$}
\label{subsec: 1_2}
Let $m = 2$. Then $T_n = (\textbf 1_2,n)$ is tridiagonal and {Hermitian}, and hence is positive definite. For example, taking $n = 2,5$ yield the corresponding matrices
$$
\begin{bmatrix}
2 & 1\\1&2
\end{bmatrix} \qquad \mbox{and}  \qquad
\begin{bmatrix}
2 & 1 & 0  & 0&0\\1&2&1&  0&0\\0&1&2&1&0\\0&0&1&2&1\\0&0&0&1&2
\end{bmatrix}.$$
$T_n$ has associated symbol $\varphi(z) = z^{-1}(z+1)^2 = 2 + z + z^{-1}$. On $\mathbb S^1$, $\varphi(z)=2 + z + \bar z = 2\cdot (1 + \operatorname{Re}(z))$ is {real-valued}.

We next consider applications of Theorems \ref{thm:GSmom} and \ref{thm:BS} for this ensemble. For $f(z) = z$, then 
\begin{align*}
2 &= \frac1n \operatorname{Tr} T_n = \frac1{2\pi} \int_0^{2\pi} \varphi(e^{i\theta}) \ d\theta 
=\frac1{2\pi} \int_0^{2\pi}2 \cdot (1 + \cos \theta) \ d\theta
\end{align*}
holds for all $n$. For $f(z) = z^p$, then we have the asymptotic result 
\begin{align}\label{eq:1_2 p}
\lim_{n\to\infty} \frac1n \sum_{j=1}^n \lambda_j(T_n)^p = \frac1{2\pi} \int_{0}^{2\pi} f(\varphi(e^{i\theta})) \ d\theta = \binom{2p}p
\end{align}
from \Cref{thm:GSmom}. To compare both sides in \eqref{eq:1_2 p} for $p = 3$, where $\binom{2p}p = 20$, we compute explicit values for $n = 10^k$ for $k = 2,3,4$:
\begin{itemize}
    \item For $n=100$, then $\frac1n \sum_{j=1}^n \lambda_j(T_n)^3 = 19.88 $

    \item For $n=1000$, then $\frac1n \sum_{j=1}^n \lambda_j(T_n)^3 = 19.988 $ 

    \item For $n=10,000$, then $\frac1n \sum_{j=1}^n \lambda_j(T_n)^3 = 19.9988$
\end{itemize}
{These trials suggest the convergence rate from \Cref{thm:GSmom} is $O(\frac1n)$, which can easily be verified for this explicit case (e.g., $T_n^3$ has fixed diagonal except only for its first and last two entries).}

A similar application then for Theorem \ref{thm:BS} using the entire function $f(z) = e^z \in C(\mathbb R)$ yields
$$\lim_{n\to\infty} \frac1n \sum_{j=1}^n f(\lambda_j(T_n)) = \frac1{2\pi} \int_0^{2\pi} f(\varphi(e^{i\theta})) \ d\theta = e^2 I_0(2) \approx 16.84398.$$ ($I_n(z)$ is the modified Bessel function of the first kind.) Now comparing this to computed explicit values for $n = 10^k$ for $k = 2,3,4$:
\begin{itemize}
    \item For $n=100$, then $\frac1n \sum_{j=1}^n f(\lambda_j(T_n)) \approx 16.7344 $

    \item For $n=1000$, then $\frac1n \sum_{j=1}^n f(\lambda_j(T_n)) \approx 16.8330 $

    \item For $n=10,000$, then $\frac1n \sum_{j=1}^n f(\lambda_j(T_n)) \approx 16.8429 $ 
\end{itemize}
{These similarly suggest a fast convergence slower than the fixed {$p^{th}$} moment case.}
\begin{figure}[t] 
\centering
    \includegraphics[width=0.6\textwidth]{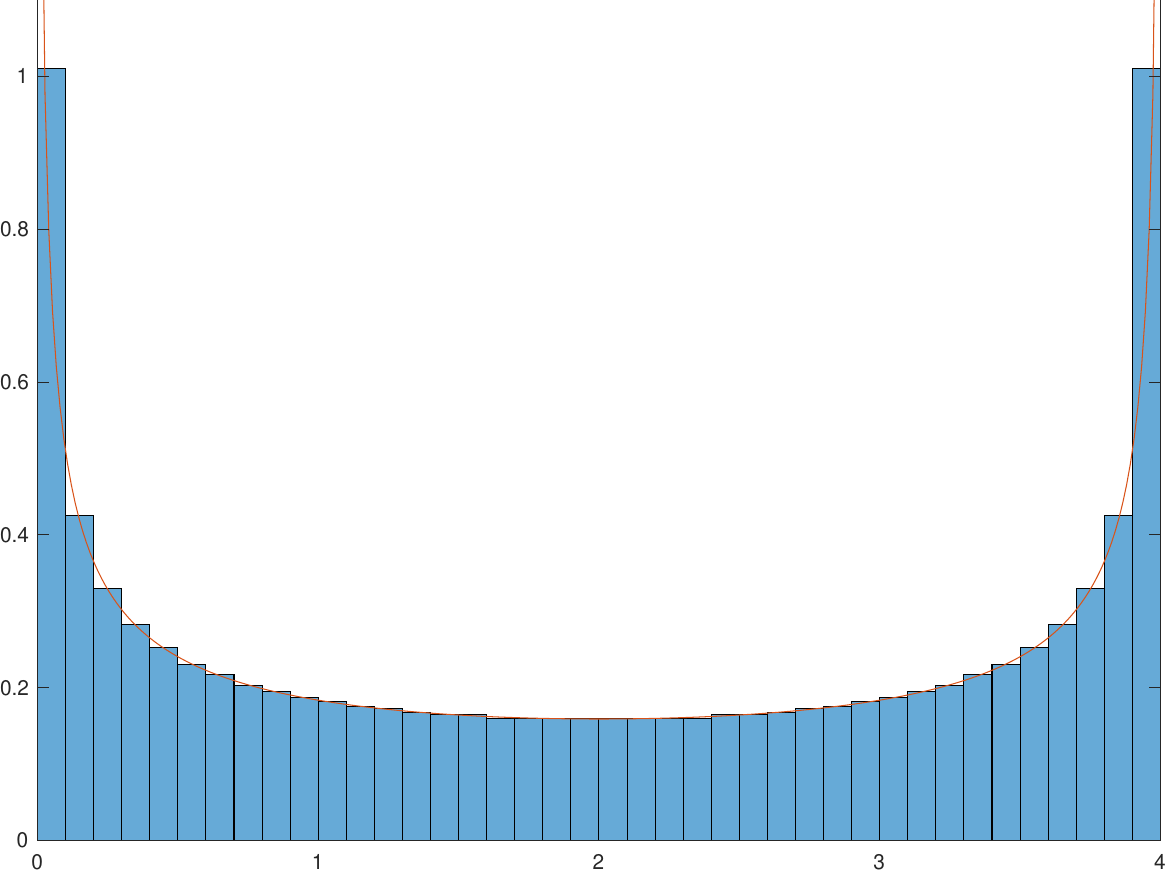}%
\caption{Histogram of (exact) eigenvalues of $(\textbf 1_2,n)$ for $n=4000$ mapped against $f(t)=(1/\pi)/\sqrt{t(4-t)}$ (in red).}
\label{fig:m=2 eigs}%
\end{figure}

 From Theorem \ref{thm:BS}, we can fully realize the ESD of the $n = 2$ case via the push-forward of the uniform map through the symbol, as seen in \Cref{fig:m=2 eigs}. 
This is possible as a consequence of the corresponding weak limit asymptotic result from Theorem \ref{thm:BS}, since the Toeplitz operator in this case is Hermitian (and positive definite). For us this only applies when $m=2$ for $(\textbf 1_m,n)$, as seen starting for $m \ge 3$ then $(\textbf 1_m,n)$ is no longer Hermitian.



\subsubsection{TPHT using $(\textbf 1_3,n)$}
For fixed $m\ge 3$, then $T_n = (\textbf 1_m, n)$ has the associated symbol $\varphi(z) = z^{-1}(z + 1)^m$ (which is \textit{not} real-valued on $\mathbb S^1$) so that 
$$
\varphi(z)^p \cdot \frac1z = z^{-p-1}(z+1)^{mp} = \sum_{k=0}^{mp} \binom{mp}k z^{k-p-1}.
$$
By \eqref{pmom},
\begin{equation}\label{eq:pth moment all 1s m-inputs}
    \lim_{n\to\infty} \frac1n \sum_{j=1}^n \lambda_j(T_n)^p = \frac1{2\pi i}\int_{\mathbb S^1} \varphi(z)^p \ \frac{dz}z = \binom{mp}p.
\end{equation}
For instance, for $m=p = 3$, then $\binom{mp}p = 84$. Comparing this to fixed computed values for again $n = 10^k$ for $k = 2,3,4$:



\begin{itemize}
    \item For $n=100$, then $\frac1n \sum_{j=1}^n \lambda_j(T_n)^3 = 83.4 $

    \item For $n=1000$, then $\frac1n \sum_{j=1}^n \lambda_j(T_n)^3 = 83.94 $

    \item For $n=10,000$, then $\frac1n \sum_{j=1}^n \lambda_j(T_n)^3 = 83.994 $  
\end{itemize}
Now for entire $f(z) = e^z$, then for $m=3$, applying Theorem \ref{thm:BS}  yields

\begin{align}\label{eq: gen 1m}
\frac1{2\pi} \int_0^{2\pi} f(\varphi(e^{i\theta})) \ d\theta  &= \sum_{k \ge 0} \frac{\binom{mk}k}{k!} \approx 169.249
{=}\lim_{n\to\infty} \frac1n \sum_{j=1}^n f(\lambda_j(T_n)).
\end{align}
This again similarly compares for computed values with fixed $n = 10^k$ for $k = 2,3,4$:
\begin{itemize}
    \item For $n=100$, then $\frac1n \sum_{j=1}^n f(\lambda_j(T_n)) \approx  166.85865$

    \item For $n=1000$, then $\frac1n \sum_{j=1}^n f(\lambda_j(T_n)) \approx  169.01002$

    \item For $n=10,000$, then $\frac1n \sum_{j=1}^n f(\lambda_j(T_n)) \approx  169.22516$
\end{itemize}


\begin{remark}
    {A followup study {might examine }explicit properties for the associated limiting spectral measure associated with $(\textbf 1_m,n)$. This can include deriving a closed form for the limiting spectral measure, as realized through an associated Hypergeometric function for the associated Laplace transform along with deriving explicit convergence rates, which the preceding examples indicate to be $O(\frac1n)$ for the fixed moment cases.}
\end{remark}

\section{Spectral theory of TPH exemplified through TPHT} \label{numspec}
The results up to this point show that the spectrum of an element of TPHT coincides with that for  all elements in the isospectral class $\mathcal{O}_\Lambda$ of TPH containing that Toeplitz element. Thus spectral results about a TPHT element hold for the full class $\mathcal{O}_\Lambda$ it represents.   (We remind that while the eigenvalues are the same, the eigenvectors within this class change while still preserving the general oscillation properties described in Section \ref{TPH}.)

\subsection{Numerics: Motivation}

We will consider the asymptotic spectral properties of TPHT matrices with random symbols. The motivation for this comes from a number of directions. For the first of these we recall that
in \cite{DE02}, Dumitriu and Edelman 
showed that the GOE($n$) random matrix ensemble ($n \times n$ symmetric matrices with independent normal entries modulo the symmetry) has a Householder tridiagonalization whose  entries are independent with distributions
\begin{eqnarray*}
\alpha_j \sim  \mathcal{N}(0,2) \qquad \mbox{and} \qquad \beta_j  \sim \chi^2_{n-j},
\end{eqnarray*}
where $\alpha_j$ are the diagonal entries of the tridiagonalization and  the symmetric entries just above and below the diagonal are distributed as $\sqrt{\beta_j}$. 
They further showed that this induces, on the eigenvalues of $X \sim \operatorname{GOE}(n)$, a joint probability distribution of the form (up to a normalization)
\begin{eqnarray} \label{DE}
P(\lambda_1, \dots,\lambda_n) = \ \exp\left( -\frac14 \sum_{k=1}^n \lambda_k^2 \right) \prod_{i < j} |\lambda_j -\lambda_i| .
\end{eqnarray}
It is a straightforward exercise to check that this distribution is the Radon-Nikodym derivative of  the invariant measure associated to the Hamiltonian dynamics of the classical Toda Lattice. This observation has led to a growth of focus on the Toda lattice as a model for hydrodynamic limits of more general lattice systems (e.g., \cite{RRV11,Spohn21}). 
The ensemble studied in this paper provides 
a basis for studying  generalized hydrodynamics for the integrable full Toda  lattice systems (cf. \cite{ER23}). 

In another direction, Freeman Dyson introduced a matrix model related to the GUE($n$) random matrix ensemble  ($n \times n$ Hermitian matrices with independent complex normal entries, modulo the Hermitian symmetry) \cite{Dyson62}. His generalization amounted to replacing the normal entries by Brownian motions. The resulting ensemble/process is nowadays referred to as Dyson Brownian motion (DBM) with principal interest being in the process it induces on eigenvalues.  This is described by the following system of stochastic ODE's.
\begin{eqnarray} \label{DBMsde}
d \lambda_k(t) = dB_k + \sum_{j \ne k} \frac1{\lambda_k(t) - \lambda_j(t)} dt, \,\,\,\, k = 1, \dots, n.
\end{eqnarray}
DBM has been intensively studied over the past few decades. If the eigenvalues are initially ordered as $ \lambda_1 > \cdots > \lambda_n$, then, with probability 1, this ordering will be maintained under the evolution. One celebrated result is that the 
distribution of the largest eigenvalue, $\lambda_1$, limits as $n \to \infty$ to the Tracy-Widom distribution,  
which is built from a particular solution of the Painlev\'e II equation \cite{TW01}. A connection for all of this with Toda 
was found by  O'Connell, 
who showed, effectively, that (\ref{DBMsde}) also arises as the zero temperature 
limit of a Markovian stochastic process whose infinitesimal generator comes from the quantum Toda lattice \cite{OConnell12}. The stochastic process considered by O'Connell is naturally expressed in terms of TPH matrices with random entries that are of log-normal type. For that reason we will consider here cases where the coefficients $a_i$ in the symbol $\varphi$ are independent, positive random variables.   We will, in particular, consider the case where the $a_i$ are independent log-normal. 
\subsection{Numerics with log-normal symbol coefficients}

We begin with the general case where the coefficients of the Toeplitz symbol, $a_i$, are independent random variables.
For such random symbols, it is of interest to compute distributional properties for the associated $p$-part of the random symbol coefficients.
{As was seen in \eqref{pmom}-\eqref{coeffs} this should correspond to the asymptotic 
limit for the $p^{th}$ moment of the ESD for the Toeplitz matrices associated to the random symbol.}
For example,  the first moment computations for the $p$-part of the symbol 
$\varphi(t)$ has mean $\mu_p^{(m)}$ given by
\begin{align}
  \mu_p^{(m)} &= \mathbb{E}\left(\sum_{\sum i_j = p} {p \choose i_1} \cdots {p \choose i_m} a_1^{i_1} \cdots
 a_m^{i_m}  \right). 
 = \sum_{\sum i_j = p} {p \choose i_1} \cdots {p \choose i_m} \mathbb{E}(a_1^{i_1}) \cdots
 \mathbb{E}(a_m^{i_m})
 \nonumber
 \end{align}

As already mentioned, we will primarily specialize to the case where the $a_i$ are log-normal. Recall that a random variable $X$ is called log-normal if 
$\log X$ is normally distributed. The density function is 
\[
\frac{1}{x \sigma \sqrt{2 \pi}} \exp{\frac{(\log x - \mu)^2}{2 \sigma^2}} 
\]
where $\mu$ and $\sigma^2$ are the mean and variance of the underlying normal distribution.

We will restrict attention to the case where $\mu = 0$ for all the random variables considered but the $\sigma$
may vary from one random variable to another.  
It follows that the terms,  $a_1^{i_1} \cdots a_m^{i_m}$, in the summand in
\eqref{coeffs} are log-normal with underlying parameters $\mu = 0$ and
$\sigma^2 = \sum_{j=1}^m i_j^2 \sigma_j^2$ where $\sigma^2_j$  is the underlying variance of $a_j$; however,  \eqref{coeffs} is then a linear combination of (dependent) log-normals but is not itself log-normal.\\

More precisely, in the case where the $a_i$ are independent log-normals with underlying parameters $\mu = 0$ and variance $\sigma_i^2$, one has
 \begin{align}
 \mu_p^{(m)}&= \sum_{\sum i_j = p} {p \choose i_1} \cdots {p \choose i_m} \exp\left( \frac{i_1^2 \sigma_1^2}{2}\right) \cdots \exp\left( \frac{i_m^2 \sigma_m^2}{2}\right)\\ 
 &=\sum_{\sum i_j = p} {p \choose i_1} \cdots {p \choose i_m} \exp\left( \frac{\sum_{j=1}^m i_j^2 \sigma_j^2}{2}\right).
\end{align}
In the independent and identically distributed (iid) log-normal case, where all $\sigma_i = \sigma$ are equal, this reduces to
\[ 
\mu_p^{(m)} = \sum_{\sum i_j = p} {p \choose i_1} \cdots {p \choose i_m}  \exp\left( \frac{\sigma^2 \sum_{j=1}^m i_j^2}{2}\right). 
\]
Using standard norm equivalence relations for $\sum_{j=1}^m i_j = p$, we have the inequalities
\[
\frac{p^2}m =\frac1m \left(\sum_{j=1}^m i_j \right)^2 \le  \sum_{j=1}^m i_j^2 \le \left(\sum_{j=1}^m i_j \right)^2 = p^2
\]
Using also the fact
\begin{equation}\label{eq:all1s moment}
    \binom{mp}{p} = \sum_{\sum i_j = p} {p \choose i_1} \cdots {p \choose i_m},
\end{equation}
(this follows from a trivial combinatorial argument of counting the number of ways of placing $p$ balls into $mp$ bins in two ways) which gives the explicit expected $p^{th}$ part coefficient for the right hand side (RHS) of the ESD  limit of $(\boldsymbol{1}_m,n)$ (see  \eqref{eq:pth moment all 1s m-inputs}), we have
\begin{equation}\label{eq:bounds}
    \binom{mp}{p} \exp\left(\frac{\sigma^2p^2}{2m}\right) \le \mu_p^{(m)} \le \binom{mp}{p} \exp\left(\frac{\sigma^2p^2}{2}\right).
\end{equation}
These comprise the upper and lower bounds found in Figures \ref{fig:rand moments diff 3-5} to \ref{fig:rand moments diff 10-20} using $\sigma = 1$ {(note the logarithmic scaling then skews the location of the mean relative to the median)}; the other (even lower) bound shown is the associated limiting spectral moment for $(\textbf 1_m,n)$ from \eqref{eq:pth moment all 1s m-inputs} and \eqref{eq:all1s moment}, which appears to be a good estimator for the sample median.

Note the above lower bound could also be achieved using a Lagrange multiplier method to minimize the objective function $f(i_1,\ldots,i_m) = \sum_{j=1}^m i_j^2$ given the constraint $\sum_{j=1}^m i_j = p$. A similar computation now using different $\sigma_1,\ldots,\sigma_m$ (so a new objective function $f(i_1,\ldots,i_m) = \sum_{j=1}^m i_j^2\sigma_j^2$) yields the lower bounds
\begin{equation}
    \binom{mp}{p}  \exp\left(\frac{p^2}{2 \sum_{j=1}^m 1/\sigma_j^{2}}\right) \ge \binom{mp}{p} \exp\left(\frac{p^2 \min_i \sigma_i^2}{2m} \right), 
\end{equation}
while a trivial upper bound holds with
\begin{equation}
    \binom{mp}p \exp\left(\frac{p^2 \max_i \sigma_i^2}2 \right).
\end{equation}
(It is straightforward to also calculate the variance of these $p$-part random symbol coefficients, but we will  not make use of that here.)
\medskip

 
 \subsubsection{Numerical experiments}
 We now turn to the numerical simulation of the $p^{th}$ moments of the spectrum of $(\textbf a_m, n)$ for large $n$ and where $\textbf a_m \in \mathbb R^m$ with iid standard log-normal components $a_i$. We match both the large $n$ matrix values from the left hand side (LHS) of Theorem \ref{thm:GSmom} against the asymptotic distributions  that are given by the RHS. We also sample the case where the $a_i$ are iid exponential distributions with unit mean. All experiments are run in MATLAB with double precision (i.e., machine precision using $\varepsilon_{\operatorname{machine}} = 2^{-52} \approx 2.2 \cdot 10^{-16}$). 

 For our experiments, we run 100,000 samples of both the LHS and RHS distributions associated with the $p^{th}$ random symbol coefficients for $(\textbf a_m, n)$, where the vector $\textbf a_m$ are generated using built-in MATLAB functions to generate normal and exponential vectors (e.g., \texttt{exp(randn(3,1))} is a standard log-normal vector in $\mathbb R^3$). To ease the following discussion, we focus our experiments on using only $m = 3,10$ input parameters and moments $p = 3,20$, as we feel these are representative of performance with other fixed combinations. 
 
 To sample the LHS matrix ensemble, we compute $\frac1n \operatorname{Tr}(A^p)$ for iid $n\times n$ $A \sim (\textbf a_m,n)$, where $A$ is formed using custom MATLAB code that generates a Toeplitz Hessenberg banded matrix whose input diagonals are the elementary symmetric polynomials associated with the random symbol vector $\mathbf a_m$. This follows since $\frac1n \operatorname{Tr}(A^p) = \frac1n \sum_{k=1}^n \lambda_k(A)^p$. To simplify discussion, we fix $n = 100$ for our experiments; empirically, the convergence in Theorem \ref{thm:GSmom} is generally fast (cf. discussion below that compares LHS and RHS empirical cumulative distribution functions (CDFs) for $n = 10$), so choosing $n = 100$ is sufficient for comparisons.

 For each sample of the RHS asymptotic $p^{th}$ moment, we form a $(p+1)\times (p+1)$ matrix  that is generated by using the symbol generated with $a_i$ iid from a prescribed distribution but now interpreted as a matrix equation, replacing $t$ with $\epsilon^T = \sum_{i=1}^{p} \textbf E_{i,i+1}$, where now the $p^{th}$ part of the symbol then aligns with the $p^{th}$ lower diagonal. For example, we form $B = \prod_{i=1}^n (\textbf I + a_i \epsilon) \in M_{p+1}(\mathbb R)$, and then store only $B_{p+1,1}$ for each sample.

 \begin{remark}\label{rmk: 2 methods}
    There is a choice on how to sample the right-hand side of Theorem \ref{thm:GSmom} in comparison to the LHS. Forming $B$ as outlined above, one could simultaneously sample both the LHS and RHS samples using the same generated $a_i$ values, or each side can be sampled independently. Figures \ref{fig:rand moments diff 3-5} to \ref{fig:rand moments diff 10-20} choose independent samples for each side.
\end{remark}

 Figures \ref{fig:rand moments diff 3-5} to \ref{fig:rand moments diff 10-20} show the summary histogram output for the 100,000 samples on a logarithmic scale for each combination of $m = 3,10$ (inputs), $p = 5,20$ (moments) and $a_i$ iid (standard log-normal versus exponential with unit mean). For comparison for each model, the associated asymptotic $p^{th}$ moment for $(\textbf 1_m,n)$, i.e., $\binom{mp}p$, is shown with a vertical yellow line. Also included are both  bounds from \eqref{eq:bounds} (i.e., $\binom{mp}p \exp\left(\frac{p^2}{2m}\right)$ and $\binom{mp}p \exp\left(\frac{p^2}{2}\right)$), which contain the mean in the iid standard log-normal case; the upper bound is omitted in the $p = 20$ cases.
\begin{figure}[htp] 
\centering
    \subfloat[Standard log-normal]{%
        \includegraphics[width=0.45\textwidth]{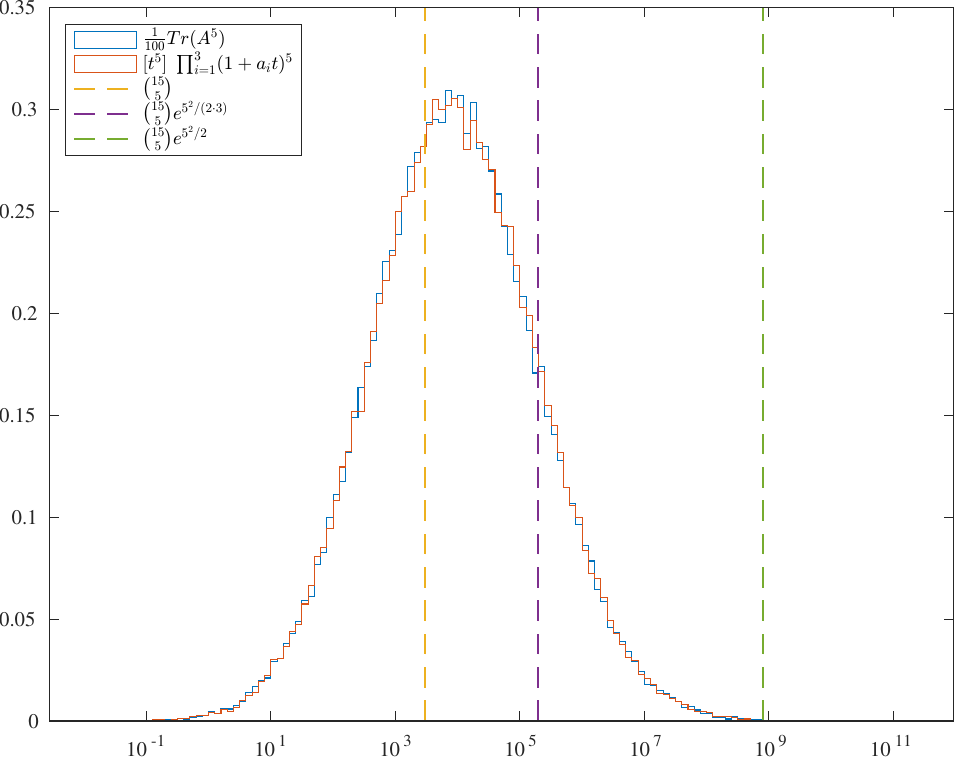}%
        \label{fig:log-normal 3-5}
        }%
    \subfloat[Exponential]{%
        \includegraphics[width=0.45\textwidth]{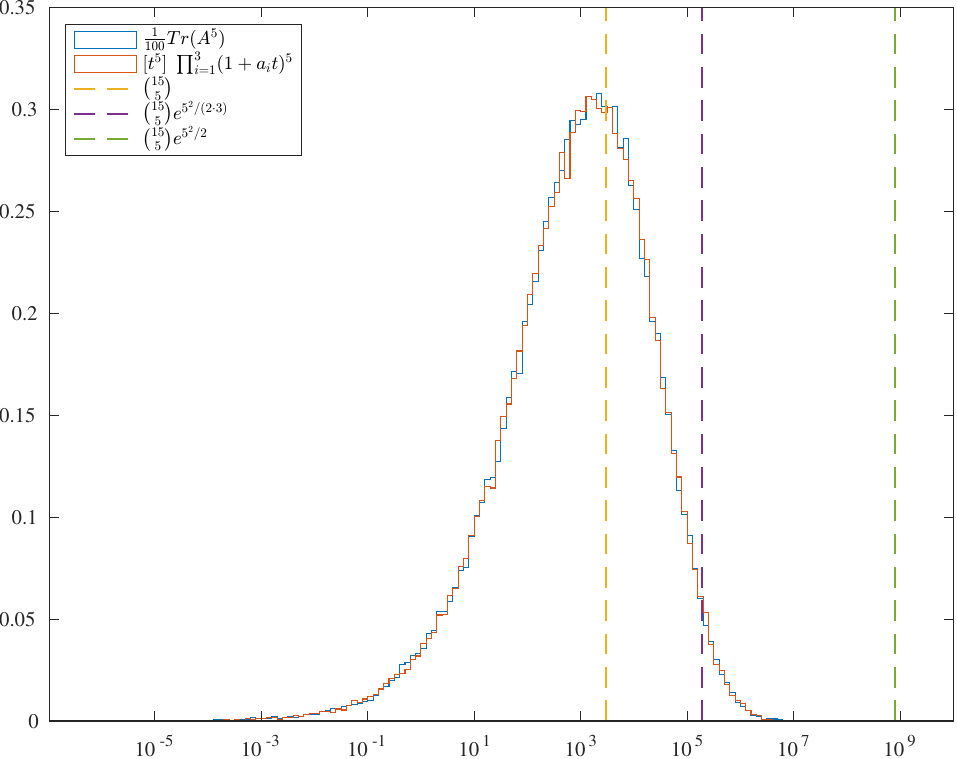}%
        \label{fig:exp 3-5}
        }%
\caption{Computed $5^{th}$ moment for ESD for $A \sim (\textbf a_3,100)$, where $\textbf a_3 \in \mathbb R^{3}$ has iid (a) standard log-normal or (b) exponential with unit mean components, using 100,000 samples}
\label{fig:rand moments diff 3-5}%
\end{figure}
\begin{figure}[htp] 
\centering
    \subfloat[Standard log-normal]{%
        \includegraphics[width=0.45\textwidth]{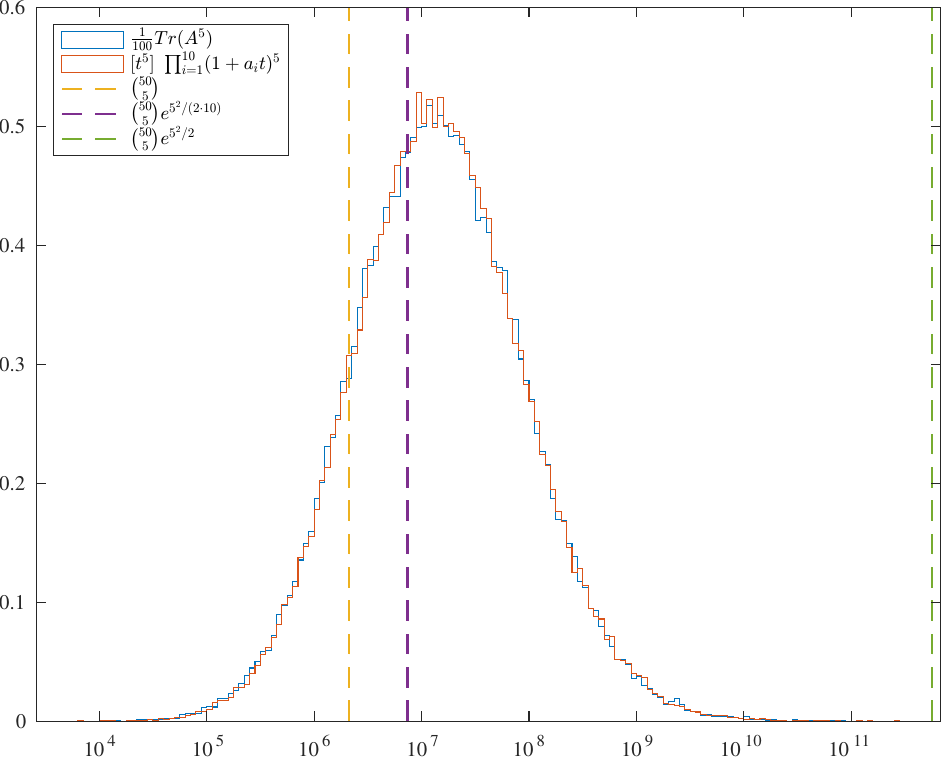}%
        \label{fig:lognorml 10-5}
        }%
    \subfloat[Exponential]{%
        \includegraphics[width=0.45\textwidth]{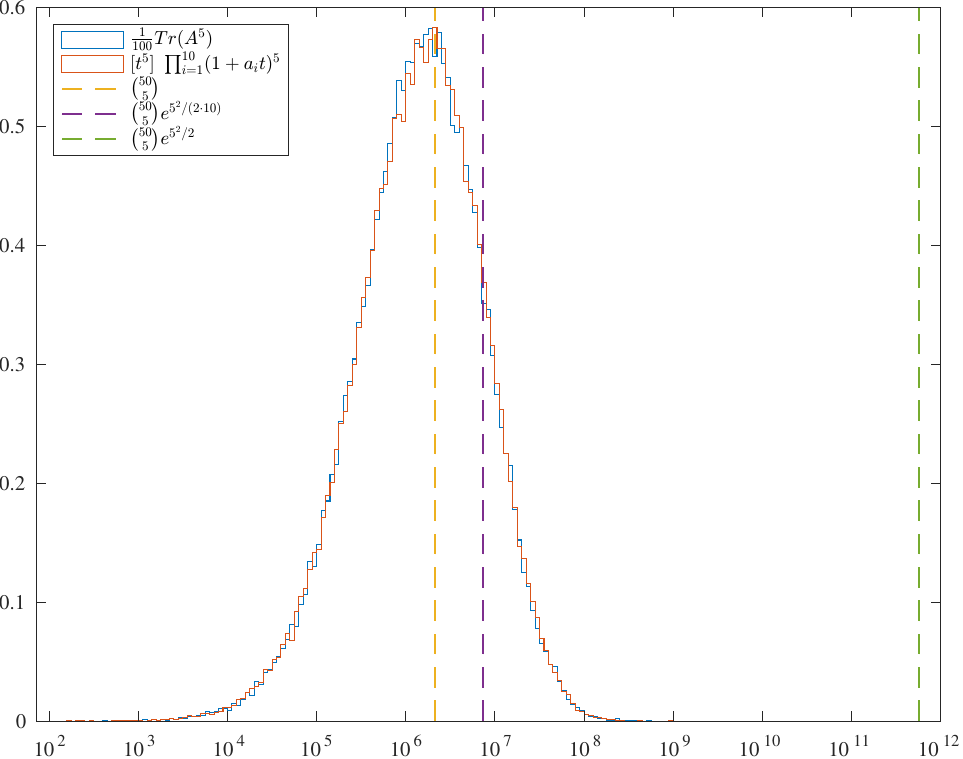}%
        \label{fig:exp 10-5}
        }%
\caption{Computed $5^{th}$ moment for ESD for $A \sim (\textbf a_{10},100)$, where $\textbf a_{10} \in \mathbb R^{10}$ has iid (a) standard log-normal or (b) exponential with unit mean components, using 100,000 samples}
\label{fig:rand moments diff 10-5}%
\end{figure}
\begin{figure}[htp] 
\centering
    \subfloat[Standard log-normal]{%
        \includegraphics[width=0.45\textwidth]{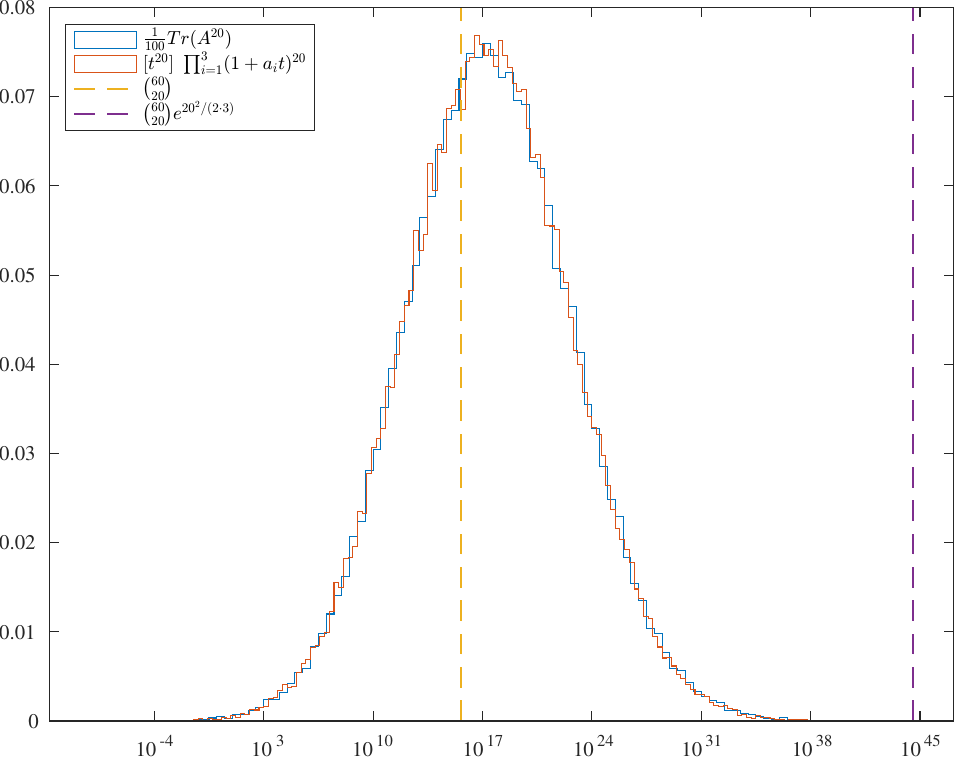}%
        \label{fig:log-normal 3-20}
        }%
    \subfloat[Exponential]{%
        \includegraphics[width=0.45\textwidth]{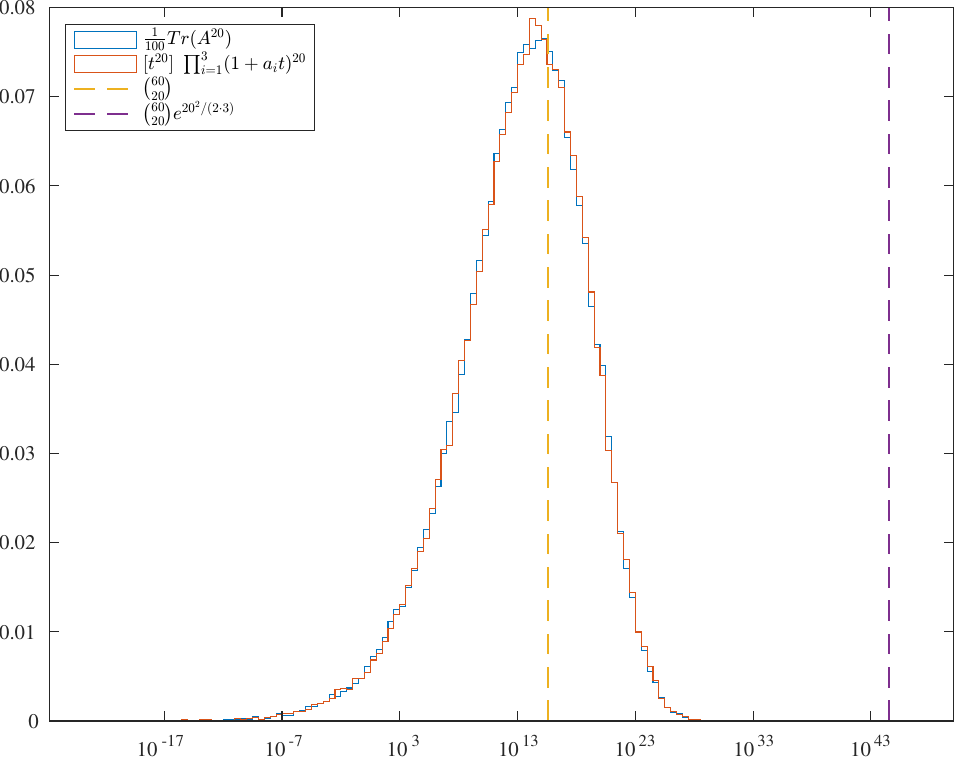}%
        \label{fig:exp 3-20}
        }%
\caption{Computed $20^{th}$ moment for ESD for $A \sim (\textbf a_{3},100)$, where $\textbf a_{3} \in \mathbb R^{3}$ has iid (a) standard log-normal or (b) exponential with unit mean components, using 100,000 samples}
\label{fig:rand moments diff 3-20}%
\end{figure}
\begin{figure}[htp] 
\centering
    \subfloat[Standard log-normal]{%
        \includegraphics[width=0.45\textwidth]{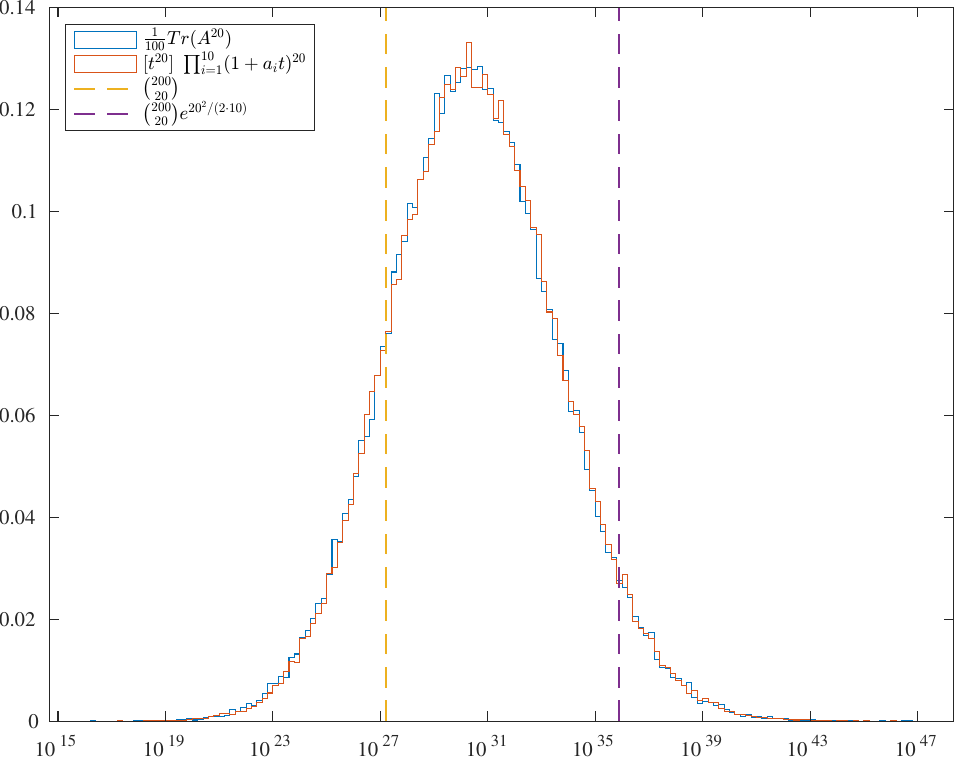}%
        \label{fig:log-normal 10-20}
        }%
    \subfloat[Exponential]{%
        \includegraphics[width=0.45\textwidth]{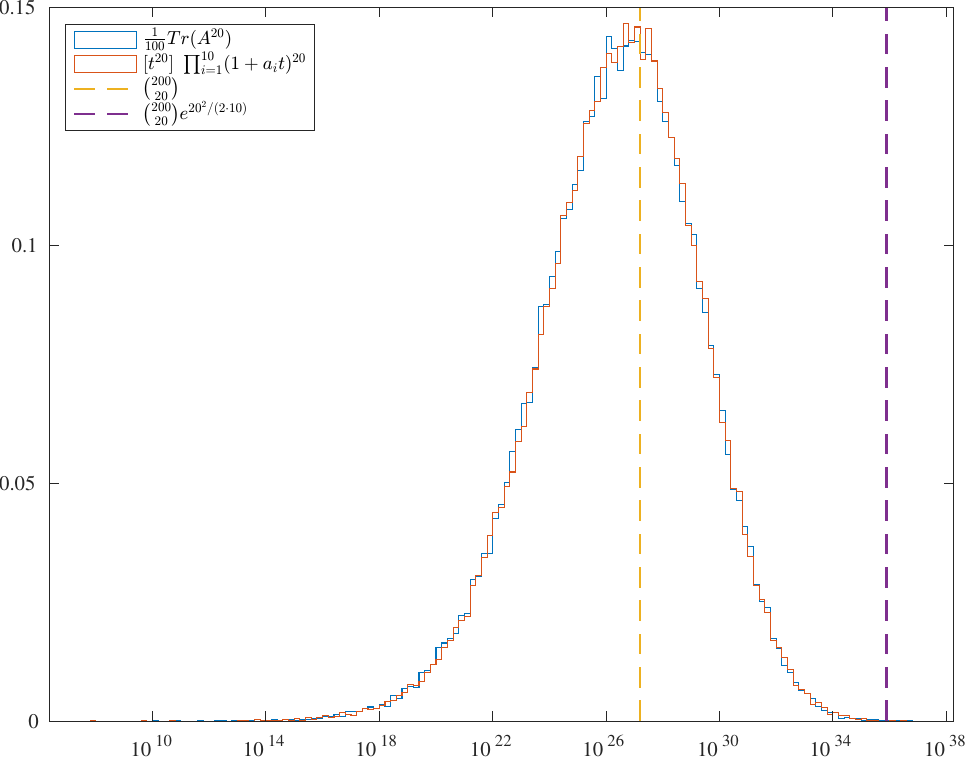}%
        \label{fig:exp 10-20}
        }%
\caption{Computed $20^{th}$ moment for ESD for $A \sim (\textbf a_{10},100)$, where $(\textbf a_{10} \in \mathbb R^{10}$ has iid (a) standard log-normal or (b) exponential with unit mean components, using 100,000 samples}
\label{fig:rand moments diff 10-20}%
\end{figure}


One quick take away for the standard log-normal maps from \Cref{fig:rand moments diff 3-5,fig:rand moments diff 10-5,fig:rand moments diff 3-20,fig:rand moments diff 10-20} is the associated $p^{th}$ moments still exhibit log-normal behavior, as seen by a near normal curve using a logarithmic scaling. For comparison, the associated exponential maps show skewed behavior on the same scaling, which tend to also approach log-normal behavior for larger moments and number of input parameters (cf. \Cref{fig:exp 10-20}). This suggests that even though no models are equal in distribution to a log-normal (even in the standard log-normal setup, then the associated $p^{th}$ moments are linear combinations of dependent log-normals), a universal behavior seems to limit toward a log-normal scheme. A followup study could focus on expanding these empirical findings.

For the $(\textbf 1_m,n)$ moment used for comparison in \Cref{fig:rand moments diff 3-5,fig:rand moments diff 10-5,fig:rand moments diff 3-20,fig:rand moments diff 10-20}, the $o(1)$ terms from the statement of Theorem \ref{thm:GSmom} can be explicitly realized as $O(\frac1n)$, as discussed previously. For example, for $A$ one of the $100\times100$ matrices (out of $10^5$ total) used to generate the exponential LHS picture in \Cref{fig:exp 3-5}, then $A^3$ has diagonal entries 1.1635 for indices 3 to 98, 1.1342 for indices 2 and 99, and 0.7202 for indices 1 and 100. Standard tools and nonparametric statistical tests  can be used to compare the distributional properties for the LHS and RHS. For instance, the Kolmogorov-Smirnov (KS) distance between each empirical CDF (i.e., supnorm distance between both empirical CDFs) is a standard comparison tool for two  distributions. For reference, we can consider 100,000 samples for each corresponding side of the $5^{th}$ moment with 3-inputs and $100\times100$ matrices with iid standard log-normal entries using both potential sampling methods (cf. \Cref{rmk: 2 methods}). When comparing the LHS to the simultaneously sampled RHS random symbol, then the KS distance is 0.00185; when comparing the LHS to the independently sampled RHS random symbol (as in \Cref{fig:log-normal 3-5}), then the KS distance is 0.00263. So these are very good matches already for $n = 100$. If doing the same but now using only $10\times 10$ order matrices for the LHS, the match now has KS distances to the RHS simultaneous and independent sampling methods, respectively, of 0.01632 and 0.02062. This further justifies choosing $n = 100$ for the above experiments, as $n = 10$ already shows strong connections to the asymptotic picture.

\subsection{Numerical Issues}

 Basak, Paquette and Zeitouni show that a small perturbed Toeplitz matrix has ESD that converges to the law of the symbol in probability \cite{BPZ20}. This is exactly what we encounter using any computations in floating-point arithmetic of any fixed (non-exact) precision order. For example, this is what we see with \Cref{fig:eig1_3,fig:eig1_8} when using the built-in \texttt{eig} function in MATLAB for $(\textbf 1_m,4000)$ for $m = 3,8$ when using double precision, since the floating-point error matrix satisfies the hypotheses of their \cite[Theorem 1]{BPZ20}. With exact arithmetic, the eigenvalues of $T_n = (\textbf 1_m,n)$ are positive and distributed on the interval $[0,2^m]$; this follows since $T_n = (\textbf 1_m,n)$ has positive eigenvalues (it's TP) that further satisfy
 \begin{equation}
     \lambda_{\max}(T_n) \le \sigma_{\max}(T_n)= \|T_n\|_2 \le \|T_n\|_1 = \max_j \sum_i |(T_n)_{ij}| = \sum_{j=0}^m \binom{m}j = 2^m
 \end{equation} 
 (using the fact the $L^1$ induced matrix norm satisfies the max column sum property, while $T_n$ has binomial coefficients as its diagonal entries (cf. \Cref{subsec:1mn})).

 \begin{figure}[t] 
\centering
    \subfloat[$(\textbf 1_3,4000)$]{%
        \includegraphics[width=0.45\textwidth]{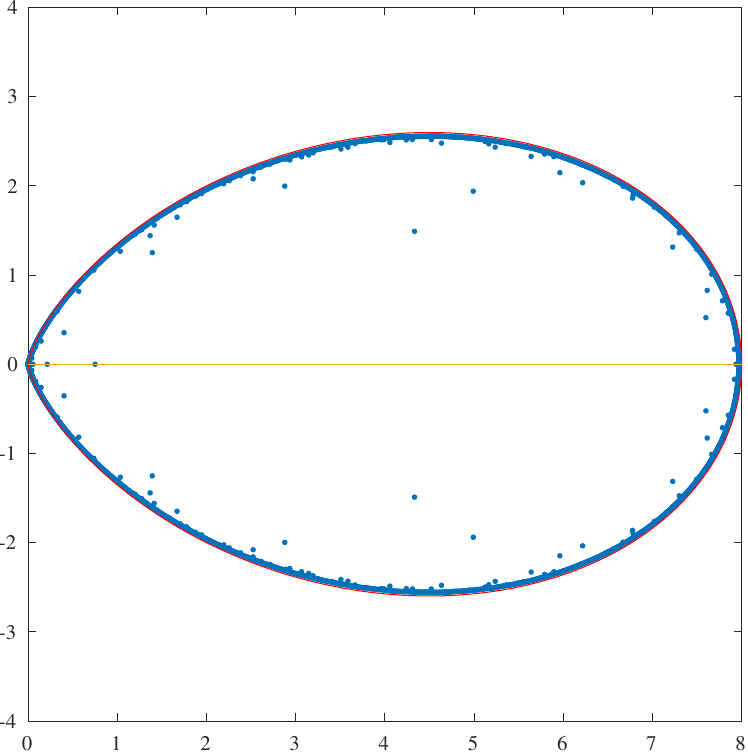}%
        \label{fig:eig1_3}
        }%
    \subfloat[$(\textbf 1_8,4000)$]{%
        \includegraphics[width=0.45\textwidth]{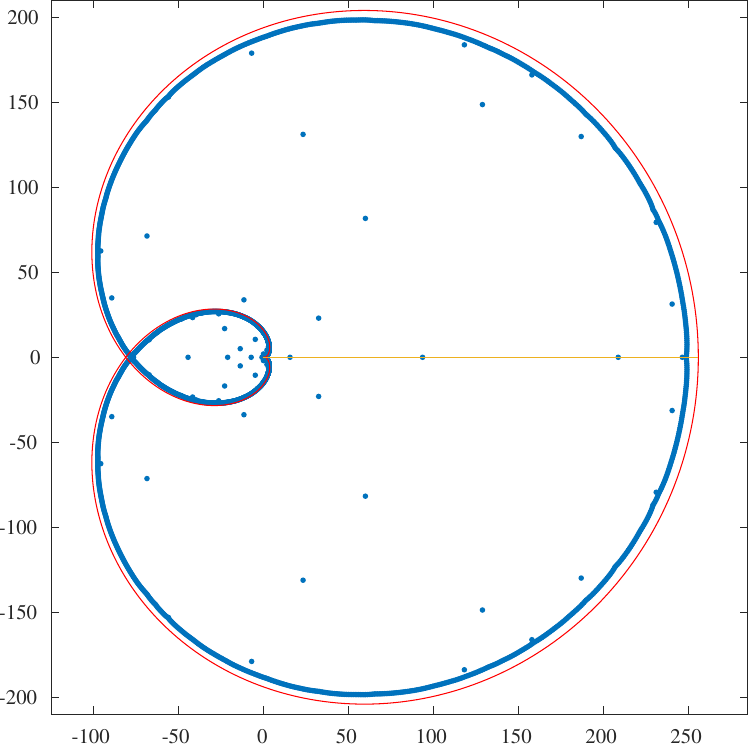}%
        \label{fig:eig1_8}
        }%
\caption{Computed eigenvalues of $(\textbf 1_m,4000)$, mapped against $\varphi(e^{i\theta})$ for $\theta \in [0,2\pi)$ where $\varphi(z) = (1 + z)^m/z$ is the (complex) symbol of the associated Toeplitz operator (in red) along with the interval $[0,2^m]$ (in yellow) for $m = 3,8$. }
\label{fig:computed errors}%
\end{figure}

Even though for our TPHT matrices we know the spectrum is real, the operator is not Hermitian for anything other than the tridiagonal case. Computations of TPHT spectra using default \texttt{eig} functions in MATLAB will similarly result in the accumulation of errors on the law of the symbol of the associated Toeplitz operator for $(\textbf 1_m,n)$ for sufficiently large $n$ when $m \ge 3$.



\section{Conclusions and Further Directions} \label{conclusions}

\subsection{TPUT}
 
    As a point of comparison, there is an elegant characterization of TP, \textit{unipotent} (bi-infinite) Toeplitz operators (TPUT) ultimately due to \cite{AESW51}. The result is

\begin{theorem}[\cite{AESW51}] \label{thm:ET}
    The symbols of all lower unipotent Toeplitz operators that are TP have precisely the form
\begin{eqnarray} \label{shift}
    \varphi(z) = 1 + x_1 z + x_2 z^2 + \cdots = e^{\gamma z} \prod_{j \ge 1} \frac{1 + a_j z}{1 - b_j z}
\end{eqnarray}   
    where $a_j,b_j$ are decreasing non-negative sequences in $\ell^1$ and $\gamma \ge 0$.
\end{theorem}
   \medskip
   
    There is also a finite size analogue of Theorem \ref{thm:ET} due to Rietsch \cite{Rietsch03}, 
    which states that 
    the class of finite TPUT matrices is
    parameterized by polynomial symbols whose coefficients are quantum elementary symmetric polynomials as opposed to the ordinary symmetric polynomials of the bi-infinite case.
    Quantum elementary symmetric polynomials here are the coefficients
    of the characteristic polynomials of Hessenberg Jacobi matrices.
    
    We note that by Theorem \ref{thm:wp}, the description of an element in $\mathcal{O}_\Lambda$ essentially reduces to a TP unipotent (TPU) element. Theorem \ref{thm:LU-T} gives a precise characterization of the unique element of this class that is Toeplitz, when that element is TP. It will be of interest to study how these representative elements may be related to the TPUT results just mentioned. 

\subsection{Total Positivity of Additional LU dynamic Invariants}
    
A remarkable property of TP matrices is that their Schur complements are also TP \cite{Ando}. This has relevance for the complete integrability of the Full Kostant-Toda Lattice (cf. \cite{EFS93} 
and \cite{E23}). 
As explained in Section \ref{TPH}, this is a generalization of the well-known tridiagonal Toda lattice whose phase space is the entirety of the lower Hessenberg matrices. As is also mentioned in that section there {is a} discrete dynamics on this phase space, consistent with the Toda dynamics and which is equivalent to the dynamics of LU factorization \cite{Symes80}. 
The eigenvalues of a Hessenberg matrix are constants of motion. Additionally, the eigenvalues of certain of its Schur complements (known as Ritz  values but also referred to as {\it chops} in the integrable systems literature \cite{EFS93}) 
are constants of motion in involution with the original eigenvalues.  The results in this paper show that for the dynamics restricted to TPH, all of these eigenvalues are real with eigenvectors having the same oscillation properties as stated in Theorem \ref{thm:osc}. 

\subsection{Lusztig Parameters}

In \cite{ER23} a presentation of the just mentioned
LU dynamics on TPH is presented in terms of the Lusztig parameters that were described in Section \ref{LUfact}. One sees from this construction that a stochastic dynamics on TPH is natural to define by taking the Lusztig parameters to be independent log-normal random variables. (See Appendix B of \cite{ER23}.) The connection between this and the stochastic structure induced from the random symbols for TPHT discussed in the current paper will be taken up in future work.  

\subsection{Comparison to Symmetric Toeplitz Operators}

It is interesting to compare our TPHT ensemble to  Hermitian Toeplitz operators for which the Toeplitz matrix is symmetric. We saw an example of where the two coincide in section \ref{subsec: 1_2}. As in that example the spectrum for Hermitian Toeplitz is always real (as for TPHT generally) but now the asymptotic density of the spectrum may be realized as the push-forward to $\mathbb{R}$ of Lebesgue measure on the circle under the symbol \cite{BG05}. 
It is also the case that there is an analogue of the Full Kostant Toda lattice in which the phase space of Hessenberg matrices is replaced by real symmetric matrices. However in this symmetric case one does not have the same tight relation to Toeplitz matrices that was described in Section  \ref{sec:HT} and that underlies our spectral analysis of general TPH class. Nevertheless, the study of symmetric Toeplitz operators merits further investigation and will also be taken up elsewhere.

\begin{figure}[tp] 
\centering
    \subfloat[Bernoulli]{%
        \includegraphics[width=0.6\textwidth]{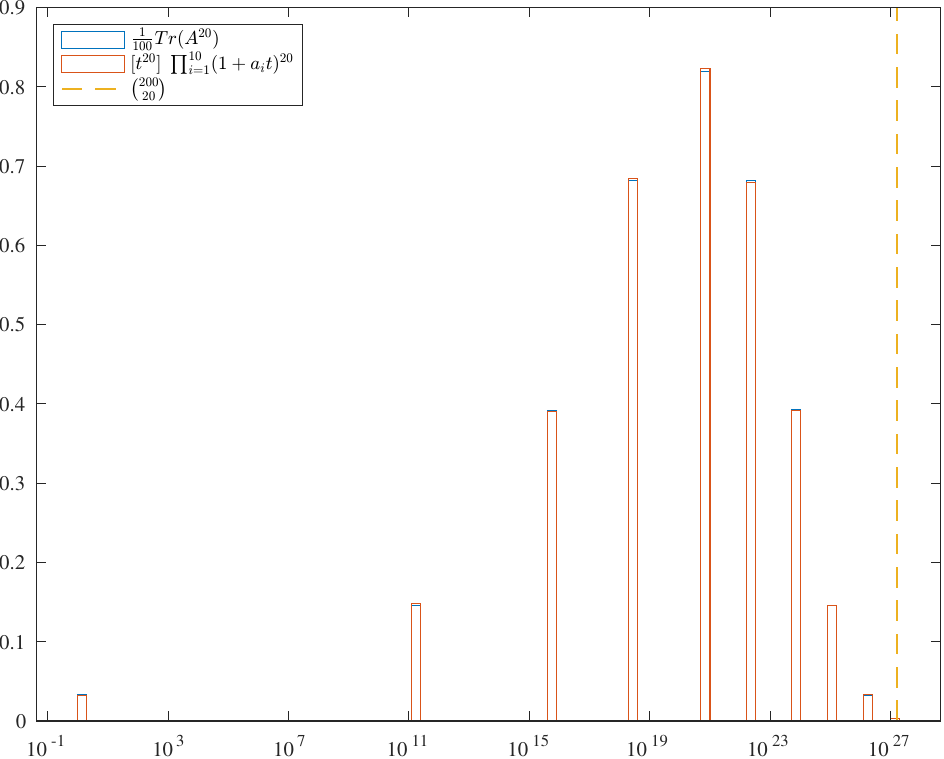}%
        }%
\caption{Computed $20^{th}$ moment for ESD for $A \sim (\textbf a_{10},100)$, where $\textbf a_{10} \in \mathbb R^{10}$ has iid Bernoulli($1/2$) components}
\label{fig:rand moments Bern}%
\end{figure}
 
\subsection{Random Symbols for Discrete Distributions}

Up to this point we have focused on random symbols associated to strictly positive continuous distributions. It is interesting to consider nonnegative discrete cases such as, for instance, $a_i \sim$ Bernoulli(q) iid. 
 If we fix $m$ inputs, then $A \sim (\textbf a_m, n)$, and $A$ is an TPHT matrix with $A = (\textbf 1_X,n)$ where $X \sim \operatorname{Binomial}(m,q)$.   In this model, the $p^{th}$ moment $\mu_p^{(m)}$ can be determined by a calculation similar to that done for $(\textbf 1_m,n)$ in Section \ref{thm:GS}. So the symbol comprises a random variable that satisfies
\begin{equation}
    \mathbb P\left(\mu_p^{(m)} = \binom{pk}{p}\right) = \mathbb P(X = k) = \binom{m}kq^k(1-q)^{m-k}.
\end{equation}
Hence, $\mu_p^{(m)} \sim \binom{pX}p$ for $X \sim \operatorname{Binomial}(m,q)$. Considering this setup, \Cref{fig:rand moments Bern} is a map using 100,000 samples of  $(\textbf a_{10}, 100)$ for $a_i \sim \operatorname{Bernoulli}(\frac12)$ compared against the limiting law from \Cref{thm:GSmom}.  Note the $X = 0$ case is excluded from the logarithmic scale histogram; of $100,000$ samples of ESD $20^{th}$ moments with input $\textbf a_{10}$, 100 resulted in $X = 0$ compared to 979 that resulted in $X = 1$ (the left-most bar in the histogram, since then $\binom{pX}p = \binom{p}p = 1 = 10^0$). The upper bound for this model matches the limiting associated moment for $(\textbf 1_{10},n)$ of $\binom{10p}p = \binom{200}{20}$, as shown by the dotted yellow line.


\appendix

\section{Appendix: Normal Forms} \label{App}

There are several normal forms that play a role  for the matrix ensembles considered in this paper. The first of these is the rational canonical normal form for general lower Hessenberg matrices: every Hessenberg matrix is conjugate to its companion matrix which is also of Hessenberg type. We will show that this conjugation is uniquely achieved by a specific lower unipotent matrix. The second normal form is a bidiagonal Hessenberg matrix which is also related to a given Hessenberg matrix  by a unique lower unipotent matrix. Finally for the more restrictive TPHT class we give, in Appendix \ref{sec:proofthm2}, the proof  of Theorem \ref{thm:LU-T} which brings into play the Hessenberg-Toeplitz normal form.

\subsection{Flag Manifolds: The Principal Embedding and the Companion Embedding}

We recall the companion matrix for matrices with spectrum $\Lambda = (\lambda_1, \ldots, \lambda_n)$ :

\begin{equation}\def\arraystretch{1.7}
c_\Lambda=\left[\begin{array}{ccccc} 
0&1\\&0&\ddots\\ &&\ddots&1\\&&&0&1\\ -c_0&-c_1& \cdots&-c_{n-2}&-c_{n-1}
\end{array}\right],
\end{equation}
where $\prod\limits_{i=1}^n(x-\lambda_i)=x^n+\sum_{i=0}^{n-1} c_ix^i$ is the characteristic polynomial for $\lambda$. \\

We also take $B_+$ to denote the group of invertible upper triangular matrices along with a distinguished 
{\it principal nilpotent} element, 
 \begin{eqnarray*}
\epsilon &=& \left(\begin{array}{ccccc}
0 & 1 & & &\\
& 0 & 1 & &\\
& & \ddots & \ddots &\\
& & & \ddots & 1\\
& & & & 0
\end{array} \right).
\end{eqnarray*}
In terms of this distinguished element we define 

\begin{equation} \label{epslam}
\epsilon_\Lambda = 
\epsilon +  \operatorname{diag}(\Lambda) =
\left[\begin{array}{cccc}
\lambda_1 & 1\\
&\lambda_2 & \ddots\\
&&\ddots & 1\\
&&&\lambda_n
\end{array}\right].
\end{equation}

\noindent In the work of \cite{ER23,OConnell13}, the focus is on $\epsilon_\Lambda$, whereas the version used by \cite{EFS93} and \cite{K78} is $c_\Lambda$. The latter has the advantage of providing a unique representative that is independent of the choice of ordering on $\Lambda$.\\

\noindent We now recall a theorem  essentially due to Kostant \cite{K78}:
\begin{theorem}[\cite{EFS93}]
For each $X\in \mathcal{O}_\Lambda$, there exists a unique lower unipotent $L\in N_-$, such that
\begin{equation}
X=L c_\Lambda L^{-1}.
\end{equation}
The same statement holds (with a different $L\in N_-$) when $c_\Lambda$ is replaced by $\epsilon_\Lambda$ (with a specific ordering of $\Lambda$).
\end{theorem}

A key feature of this result is that $L$ is unique and so makes possible the following definition. 

\begin{definition}\label{compembdefchfive}
The \textbf{companion embedding} is the map $\kappa_\Lambda:\mathcal{O}_\Lambda\to  {G}/B_+$ defined as follows: for $X\in \mathcal{O}_\Lambda$, if $X=Lc_\Lambda L^{-1}$, then
\begin{equation}
\kappa_\Lambda(X)=L^{-1}\mod B_+.
\end{equation}
\end{definition}

\begin{remark}\index{Crystal Embedding}\label{crystalembdefchfive}
An analogous embedding (described later in this section) can be performed using $\epsilon_\Lambda$ in place of $c_\Lambda$. We call this the \textbf{principal embedding}.  
\end{remark}

\noindent We now turn our attention to the $L$'s in both embeddings, finding explicit formul{\ae} where possible and offering a means of translation between the two by expressing the relationship between the $L$'s corresponding to $\epsilon_\Lambda$ and to $c_\Lambda$.

\begin{lemma}\label{explicitlowerkostant}
For each $n\in\mathbb{N}$, if $X$ is a tridiagonal $n\times n$ Hessenberg matrix, and for $1<k\leq n$, 
$$\det(xI_{k-1}-X^{(k-1)})=x^{k-1}+\sum_{i=1}^{k-1} l_{ki}x^{i-1}$$
the $n\times n$ lower unipotent matrix $L=(l_{ij})_{i,j}$ defined by the above $n-1$ polynomials is the unique such matrix satisfying
$$L^{-1}XL=c_X$$
where $c_X$ is the companion matrix of $X$ (or $c_\Lambda$, where $\Lambda=\operatorname{Spec}(X)$), and where $X^{(k)}$ denotes the principal $k \times k$ submatrix of $X$.
\end{lemma}

\begin{proof}
We prove this by induction:\\[5pt]
The base case of $n=1$ is trivial: $X=[a_1]$, $L=[1]$ and $c_X=[a_1]$ clearly satisfies $XL=Lc_X$.\\[5pt]
Let us suppose the result holds for some $n\in\mathbb{N}$. To proceed, suppose $X$ is an $(n+1)\times (n+1)$ tridiagonal Hessenberg matrix. We make the key observation that if $L$ is the conjugating matrix for $X$, then $L^{(n)}$ is the conjugating matrix for $X^{(n)}$, which follows immediately from the definition of the $l_{ki}$'s. Thus, the induction hypothesis asserts 
\begin{equation}X^{(n)}L^{(n)}=L^{(n)}c_{X^{(n)}}.\label{conjcompblocks}\end{equation}
We impose a block structure on $X$:
$$\renewcommand{\arraystretch}{1.6}\setlength{\tabcolsep}{16pt}
X=\left[\begin{array}{c|c}
X^{(n)} & \mathbf{e}_{n}\\\hline
b_{n}\mathbf{e}_{n}^T & a_{n+1}
\end{array}\right]$$
where $\mathbf{e}_n$ is the last column of the $n\times n$ identity matrix.\\

\noindent We impose the analogous block structure on $L$:
$$\renewcommand{\arraystretch}{1.6}\setlength{\tabcolsep}{16pt}
L=\left[\begin{array}{c|c}
L^{(n)} &0\\\hline
\mathbf{v}^T & 1
\end{array}\right]$$
where $\mathbf{v}=[l_{n+1,1}~l_{n+1,2}~\cdots~l_{n+1,n}]$.\\

\noindent Let $c=L^{-1}XL$ with block structure $c=\renewcommand{\arraystretch}{1.5}\left[\begin{array}{c|c}
A & \mathbf{p}\\\hline
\mathbf{q}^T  & r
\end{array}\right]$. Since $XL=Lc$, one obtains the following equation from the top-left block:
\begin{equation}X^{(n)}L^{(n)}+\mathbf{e}_{n}\mathbf{v}^T =L^{(n)}A,\label{indhypunicomp}\end{equation}
with $A$ an $n\times n$ matrix.\\

\noindent By the invertibility of $L^{(n)}$, there can be only one $A$ satisfying this equation. \\

\noindent \textbf{Claim.} $A=\epsilon_n$, where $\epsilon_n$ is the $n\times n$ matrix with $1$'s on the superdiagonal and zeroes elsewhere. This $\epsilon_n$ is not to be mistaken with $\epsilon_\Lambda$. If $d_\Lambda=\text{diag}(\lambda_1,\ldots,\lambda_n)$, then $\epsilon_\Lambda=d_\Lambda+\epsilon_n$.\\[2pt]
\textbf{Proof of Claim.} Using the induction hypothesis, and plugging in $A=\epsilon_n$, Equation \ref{indhypunicomp} becomes
\begin{equation}L^{(n)}c_{X^{(n)}}+\mathbf{e}_{n}\mathbf{v}^T =L^{(n)}\epsilon_n,\end{equation}
or, equivalently,
\begin{equation}L^{(n)}(\epsilon_n-c_{X^{(n)}})=\mathbf{e}_{n}\mathbf{v}^T.\end{equation}
Evaluating both sides, one obtains\vspace{0.2cm}
\begin{equation}\renewcommand{\arraystretch}{1.6}\setlength{\tabcolsep}{16pt}
\left[\begin{array}{cccc}
0&0&\cdots &0\\
\vdots& &&\vdots \\
0 & 0&\cdots & 0\\
l_{n+1,1}&l_{n+1,2} & \cdots & l_{n+1,n}
\end{array}\right]
=
\left[\begin{array}{cccc}
0&0&\cdots &0\\
\vdots& &&\vdots \\
0 & 0&\cdots & 0\\
p_0&p_1 & \cdots & p_{n-1}
\end{array}\right]
\end{equation}
where 
$$\det(xI_{n}-X^{(n)})=x^{n}+\sum_{k=0}^{n-1}p_kx^k=x^{n}+\sum_{i=1}^{n} l_{n+1,i}x^{i-1}
=x^{n}+\sum_{i=0}^{n-1} l_{n+1,i+1}x^i.$$
Hence, $p_k=l_{n+1,k+1}$ for $k=0,\ldots,n-1$. Thus, the ansatz of $A=\epsilon_{n}$ was consistent, which proves the claim. $\square$\\

\noindent Returning to \eqref{conjcompblocks}, we turn our attention to the top-right block:
\begin{equation}
\mathbf{e}_n=L^{(n)}\mathbf{p}.
\end{equation}
Since $(L^{(n)})^{-1}$ is lower unipotent, $\mathbf{p}=(L^{(n)})^{-1}\mathbf{e}_n=\mathbf{e}_n$, since the $\mathbf{e}_n$ is also the last column of $(L^{(n)})^{-1}$.\\

\noindent One can conclude therefore that this matrix $c$, given by $L^{-1}XL$, is a companion matrix. Since $c$ is conjugate to $X$, and the characteristic polynomial is invariant under matrix conjugation, one must have that $c$ is indeed the companion matrix for $X$. This completes the induction step, proving the theorem.
\end{proof}

\noindent This gives a means for computing $L^{-1}$ in the principal embedding.

\begin{lemma}\label{cltoepsl}
Let $X\in \mathcal{O}_\Lambda$, and let $L_1$ be defined as in Lemma \ref{explicitlowerkostant}, and let $L_2$ be the lower unipotent matrix such that for $i>j$
\begin{equation}
(L_2)_{ij}=(-1)^{i+j}e_{j-i}(\lambda_1,\ldots,\lambda_{i-1})
\end{equation}
where $e_{j-1}$ is the $(j-i)$-th elementary symmetric polynomial
\begin{equation}
e_{j-i}(\lambda_1,\ldots,\lambda_{i-1})=\sum_{1\leq k_1<k_2<\cdots<k_{j-i}\leq n}\lambda_{k_1}\lambda_{k_2}\cdots\lambda_{k_{j-i}},
\end{equation}
then $L=L_2L_1^{-1}$ satisfies $X=L^{-1}\epsilon_\Lambda L$.
\end{lemma}

\begin{proof}
This is a consequence of Lemma \ref{explicitlowerkostant}. One has $X=L_1c_\Lambda L_1^{-1}$, and I claim that $c_\Lambda=L_2^{-1}\epsilon_\Lambda L_2$. Thus, $X=L_1L_2^{-1}\epsilon_\Lambda L_2L_1^{-1}$, and so $L=(L_1L_2^{-1})^{-1}=L_2L_1^{-1}$.\\

\noindent The claim itself is simply an application of Lemma \ref{explicitlowerkostant} since
\begin{equation}
\tau_k(xI_n-\epsilon_\Lambda)=x^k+\sum_{i=0}^{k-1}(-1)^{k+i}e_{k-i}(\lambda_1,\ldots,\lambda_k).
\end{equation}
\end{proof}

\noindent When $\lambda_i\neq \lambda_j$ for all $i\neq j$, one can of course diagonalise any matrix in $\mathcal{O}_\Lambda$. The following result, which is an explicit form of Lemma 7 in \cite{EFH91}, describes a diagonalisation of $\epsilon_\Lambda$.

\begin{lemma}\label{uppertodiagefhexpl}
If $\lambda_1,\ldots,\lambda_n$ are distinct, then one has $\epsilon_\Lambda=U d_\Lambda U^{-1}$, where $U=(u_{ij})$ is the upper triangular matrix given by
$$u_{ij}=\prod_{k=1}^{i-1}(\lambda_j-\lambda_k),~~~1\leq i\leq j\leq n$$
and $d_\Lambda=\epsilon_\Lambda-\epsilon=\operatorname{diag}(\lambda_1,\ldots,\lambda_n)$.
\end{lemma}

\begin{proof}
The matrix $U$ is clearly invertible if and only if $\lambda_i\neq \lambda_j$ since
$$\det(U)=\prod_{i=1}^n \prod _{k=1}^{i-1}(\lambda_i-\lambda_k)=\prod_{1\leq k<i\leq n}(\lambda_i-\lambda_k).$$
It just remains to show that $\epsilon_\Lambda U = U d_\Lambda$. Let $u_j=(u_{ij})_{1\leq i\leq n}$ be the $j$-th column of $U$, then for $i<n$:
\begin{align*}
(\epsilon_\Lambda u_j)_i&=\sum_{k=1}^n (\epsilon_\Lambda)_{ik}u_{kj}\\
&=\lambda_i u_{ij}+u_{i+1,j}\\
&=\lambda_i\prod_{k=1}^{i-1}(\lambda_j-\lambda_k)+\prod_{k=1}^{i}(\lambda_j-\lambda_k)\\
&=(\lambda_i+\lambda_j-\lambda_i)\prod_{k=1}^{i-1}(\lambda_j-\lambda_k)\\
&=\lambda_j \prod_{k=1}^{i-1}(\lambda_j-\lambda_k)\\
&=\lambda_j (u_j)_i.
\end{align*}
For $i=n$, we simply have
\begin{align*}
(\epsilon_\Lambda u_j)_n&=\lambda_n (u_j)_n\\
&=\left\{ \begin{array}{cl} 0 & j<n\\ \lambda_j (u_j)_n & j=n\end{array}\right.\\
&=\lambda_j (u_j)_n.
\end{align*}~\\[-12pt]
Thus, $\epsilon_\Lambda u_j=\lambda_j u_j$ for each $j$. 
\end{proof}

A final feature of these embeddings is that they provide a means of representing the eigenfunctions for Hessenberg matrices $X$ when the eigenvalues of $X$ are distinct.
\begin{corollary} \label{cor:efcn}
\begin{eqnarray*}
    X  L^{-1} U = L^{-1} U d_\Lambda
\end{eqnarray*} 
where $L= L_2 L_1^{-1}$ from Lemma \ref{cltoepsl}. In other words $L^{-1} U$ is the matrix of eigenfunctions for $X$ presented in LU-factorized form. 
\end{corollary}
The  proof is an immediate consequence of the previous two lemmas.
\medskip

\section{Proof of Theorem \ref{thm:LU-T}}\label{sec:proofthm2}

We recall the statement of the theorem and give its detailed proof.\\
\medskip

\noindent\textbf{Theorem \ref{thm:LU-T}.}
\textit{Let $T$ be an $n\times n$ TPHT matrix {in $\mathcal{H}^{\geq 0}$, defined in Theorem \ref{thm:wp}.}  Then $T$ has the LU decomposition $T=LU$ where}
\begin{align*}
    (L)_{ij} &= 
    \left\{
    \begin{array}{cc}
0 & i<j\\
\frac{\tau_{\{i\}\cup[j-1]}^\text{init}(T)}{\tau_{[j]}^\text{init}(T) } & i\geq j
    \end{array}
    \right.\\
    (U)_{ij} &= 
    \left\{
    \begin{array}{cc}
\frac{\tau_{[i]}^\text{init}(T)}{\tau_{[i-1]}^\text{init}(T)} & i=j\\
1 & j=i+1\\
0 & \text{\textit{otherwise.}}
    \end{array}
    \right.
\end{align*}

\begin{proof}
We prove this by induction on $n$. To aid in the proof, for $n\in \mathbb{N}$, denote by $T_n$ the matrix given by
$$(T_n)_{ij} 
= 
\left\{
\begin{array}{cc}
a_{i-j+1} & i\geq j \\
1 & j=i+1\\
0 & j>i+1
\end{array}
\right.
= 
\left[
\begin{array}{ccccc}
a_1 & 1 & 0 & \cdots & 0\\
a_2 & a_1 & 1 & \ddots & \vdots\\
a_3 & a_2 & \ddots & \ddots & 0 \\
\vdots & \ddots & \ddots & \ddots & 1\\
a_n & \cdots & a_3 & a_2 & a_1
\end{array}
\right].$$
If $n=1$, the theorem  states that
$$T_1 = [\,a_1\,] = [\, 1\, ] \, [\, \tau_{[1]}^\text{init}(T)\,]$$
which holds trivially. So, assume $n>1$.
Now, observe that $T_n$ sits inside $T_{n+1}$ as its principal sub-matrix (top-left):
$$T_{n+1} = \left[\begin{array}{c|c}
T_n & \mathbf{e}_n\\
\hline
[a_{n+1}~\cdots ~ a_3 ~ a_2] & a_1
\end{array}\right].$$
We assume that $T_n = L_n U_n$ where 
\begin{align*}
    (L_n)_{ij} &= 
    \left\{
    \begin{array}{cc}
0 & i<j\\
\frac{\tau_{\{i\}\cup[j-1]}^\text{init}(T_n)}{\tau_{[j]}^\text{init}(T_n) } & i\geq j
    \end{array}
    \right.
    ~~~=~~~~~ 
    \left\{
    \begin{array}{cc}
0 & i<j\\
\frac{\tau_{\{i\}\cup[j-1]}^\text{init}(T_{n+1})}{\tau_{[j]}^\text{init}(T_{n+1}) } & i\geq j
    \end{array}
    \right.\\
    (U_n)_{ij} &= 
    \left\{
    \begin{array}{cc}
\frac{\tau_{[i]}^\text{init}(T_n)}{\tau_{[i-1]}^\text{init}(T_n)} & i=j\\
1 & j=i+1\\
0 & \text{otherwise}
    \end{array}
    \right.
    ~~~=~~~~~ 
    \left\{
    \begin{array}{cc}
\frac{\tau_{[i]}^\text{init}(T_{n+1})}{\tau_{[i-1]}^\text{init}(T_{n+1})} & i=j\\
1 & j=i+1\\
0 & \text{otherwise}
    \end{array}
    \right..
\end{align*}
Note the second equality in each line is due to $T_n$ sitting inside $T_{n+1}$ as its principal sub-matrix. Because of this, the rest of this proof shall write $\tau^\text{init}_S$ for $\tau^\text{init}_S(T_{n+1})$.\\
Now consider the LU decommposition of $T_{n+1}$ in $2\times 2$ block form with the principal $n\times n$ sub-matrix as a block:
$$T_{n+1} = L_{n+1}U_{n+1} = 
\left[\begin{array}{c|c}
L_{n+1}^{(n)} & \mathbf{0}_n\\
\hline
\mathbf{p}^T & 1
\end{array}\right]
\left[\begin{array}{c|c}
U_{n+1}^{(n)} & \mathbf{q}\\
\hline
\mathbf{0}_n^T & r
\end{array}\right]
=
\left[\begin{array}{c|c}
L_{n+1}^{(n)}U_{n+1}^{(n)} & L_{n+1}^{(n)}\mathbf{q}\\
\hline
\mathbf{p}^T U_{n+1}^{(n)} & \mathbf{p}^T\mathbf{q}+ r
\end{array}\right]
$$
where $A^{(n)}$ denotes the principal $n\times n$ sub-matrix.\\
By uniqueness of the LU decomposition, this implies that $L_{n+1}^{(n)}=L_n$ and $U_n = U_{n+1}^{(n)}$, which gives a nesting of LU decompositions for the sequence of matrices $(T_n)_{n\in \mathbb{N}}$.\\
To see how the LU decomposition for $T_{n+1}$ relates to that of $T_n$, we need to solve the following for $\mathbf{p}$, $\mathbf{q}$ and $r$:
$$L_n\mathbf{q}=\mathbf{e}_n,~~~~~
\mathbf{p}^T U_n=[a_{n+1}~\cdots ~ a_3 ~ a_2],~~~~~
\mathbf{p}^T\mathbf{q}+ r = a_1.
$$
We immediately have $\mathbf{q}=L_n^{-1}\mathbf{e}_n=\mathbf{e}_n$ since $L_n$ and its inverse is lower unipotent.\\
We claim $\mathbf{p}^T = [p_1 ~ p_2 ~ \cdots ~ p_n]$ with $p_j = \frac{\tau^\text{init}_{\{n+1\}\cup [j-1]}}{\tau^\text{init}_{[j]}}$ satisfies $\mathbf{p}^T U_n=[a_{n+1}~\cdots ~ a_3 ~ a_2]$. Let us multiply this out formally
\begin{align*}
    (\mathbf{p}_n^T U_n)_{1j}
    &= \sum_{k=1}^n (\mathbf{p}_n^T)_{1k}(U_n)_{kj}\\
    &=\left\{
    \begin{array}{cc}
    \left(\frac{\tau^\text{init}_{\{n+1\}}}{\tau^\text{init}_{[1]}}\right)\tau^\text{init}_{[1]}
    & j=1\\
    \frac{\tau^\text{init}_{\{n+1\}\cup [j-2]}}{\tau^\text{init}_{[j-1]}}\cdot 1 + \frac{\tau^\text{init}_{\{n+1\}\cup [j-1]}}{\tau^\text{init}_{[j]}}\cdot \frac{\tau^\text{init}_{[j]}}{\tau^\text{init}_{[j-1]}} 
    & 1<j\leq n
    \end{array}
    \right. \\
    &=\left\{
    \begin{array}{cc}
    \tau^\text{init}_{\{n+1\}}
    & j=1\\
    \frac{\tau^\text{init}_{\{n+1\}\cup [j-2]}+\tau^\text{init}_{\{n+1\}\cup [j-1]}}{\tau^\text{init}_{[j-1]}} 
    & 1<j\leq n
    \end{array}
    \right..
\end{align*}
The first part satisfies $\tau^\text{init}_{\{n+1\}\cup [j-2]}=a_{n+1}$ simply because this is the minor determinant of the $1\times 1$ sub-matrix of $T_{n+1}$ in the last row of the first column.\\
It now remains to show that
$$\frac{\tau^\text{init}_{\{n+1\}\cup [j-2]}+\tau^\text{init}_{\{n+1\}\cup [j-1]}}{\tau^\text{init}_{[j-1]}} = a_{n+2-j},~~~~~\text{for}~~1<j\leq n.$$
But note that this is equivalent to 
$$\tau^\text{init}_{\{n+1\}\cup [j-1]}
= a_{n+2-j}\cdot \tau^\text{init}_{[j-1]} -
\tau^\text{init}_{\{n+1\}\cup [j-2]}.$$
This is true because the left-hand side is given by
\begin{align*}
    \tau^\text{init}_{\{n+1\}\cup [j-1]}
    &=\det((T_{n+1})_{\{n+1\}\cup [j-1],[j]})\\
    &=\det\left[
    \begin{array}{ccccc}
    a_1 & 1 & 0 & \cdots & 0\\
    a_2 & a_1 & \ddots & \ddots & \vdots\\
    \vdots & \ddots & \ddots & \ddots & \vdots\\
    a_{j-1} & \cdots & a_2 & a_1 & 1\\
    \hline
    a_{n+1} & a_n & a_{n-1} & \cdots & a_{n+2-j}
    \end{array}
    \right]
\end{align*}
so the desired equality is seen as the cofactor expansion down the last column in the above matrix.\\
Finally, we need to show that $\mathbf{p}^T \mathbf{q}+r = a_1$ implies $r=\frac{\tau^\text{init}_{[n+1]}}{\tau^\text{init}_{[n]}}$. Since we know $\mathbf{p}$ and $\mathbf{q}$, we can compute this directly:
\begin{align*}
    r
    &=a_1-\mathbf{p}^T \mathbf{q}\\
    &= a_1-\mathbf{p}^T \mathbf{e}_n\\
    &=a_1-p_n\\
    &=a_1-\frac{\tau^\text{init}_{\{n+1\}\cup[n-1]}}{\tau^\text{init}_{[n]}}\\
    &=\frac{a_1\cdot \tau^\text{init}_{[n]}-\tau^\text{init}_{\{n+1\}\cup[n-1]}}{\tau^\text{init}_{[n]}}\\
    &= \frac{\tau^\text{init}_{[n+1]}}{\tau^\text{init}_{[n]}}
\end{align*}
where the last equation follows by considering the cofactor expansion down the last column of $T_{n+1}$.
\end{proof}

\bibliographystyle{siamplain}
\bibliography{references}

\end{document}